\documentclass[12pt]{amsart}  
\usepackage{amsmath,amstext,amsthm,amssymb,amsxtra}
\usepackage[normalem]{ulem}
\usepackage{txfonts} 
\usepackage[T1]{fontenc}
\usepackage{lmodern}

 \usepackage{euler}   

\usepackage{mathtools}
\mathtoolsset{showonlyrefs,showmanualtags}

\usepackage{hyperref} 
\hypersetup{
    colorlinks=true,       
    linkcolor=blue,          
    citecolor=magenta,        
    filecolor=magenta,      
    urlcolor=cyan           
}

\usepackage[msc-links]{amsrefs}  

\allowdisplaybreaks

\theoremstyle{plain} 
\newtheorem{lemma}[equation]{Lemma} 
\newtheorem{proposition}[equation]{Proposition} 
\newtheorem{theorem}[equation]{Theorem} 
\newtheorem{corollary}[equation]{Corollary}

\theoremstyle{definition}
\newtheorem{definition}[equation]{Definition} 

\theoremstyle{remark}

\numberwithin{equation}{section}

\title[Two Weight Riesz Transform Inequalities]{Two Weight Inequalities for Riesz Transforms:\\ Uniformly Full Dimension Weights} 

\author[M. T. Lacey]{Michael T. Lacey$^{1}$}   

\address{Michael T. Lacey, School of Mathematics, Georgia Institute of Technology, Atlanta GA 30332, USA}
\email {lacey@math.gatech.edu}
\thanks{1.  Research supported in part by grant NSF-DMS 1265570, 
and the Australian Research Council through grant ARC-DP120100399.}

\author[B. D. Wick]{Brett D. Wick$^{2}$}
\address{Brett D. Wick, Department of Mathematics\\ Washington University -- St. Louis\\ One Brookings Drive\\ St. Louis, MO USA 63130}
\email{wick@math.wustl.edu}
\thanks{2.  Research supported in part by a National Science Foundation DMS grants \# 1603246 and \# 1560955.}

\begin{document}
\begin{abstract}  
Fix an integer $ n$ and number $d$, $ 0< d\neq n-1 \leq n$, and two weights $ w$ and $ \sigma $ on $ \mathbb R ^{n}$. 
We impose an extra condition that the two weights separately are not concentrated on a set of codimension one, 
uniformly over locations and scales.  (This condition holds for doubling weights.)
Then, we characterize the two weight inequality for the $ d$-dimensional Riesz transform on $ \mathbb R ^{n}$, 
\begin{equation*}
\left\lVert 
\int   f (y) \frac {x-y} {\lvert  x-y\rvert ^{d+1} } \; \sigma (dy)
\right\rVert_{L ^{2} (\mathbb{R}^n;w)} \le \mathscr N \lVert f\rVert_{L ^2 (\mathbb{R}^n;\sigma )}
\end{equation*}
in terms of these two conditions, and their duals:  For finite constants $ \mathscr A_2$ and $ \mathscr T$, 
uniformly over all cubes $ Q\subset \mathbb R ^{n}$
\begin{gather*}
\frac {w (Q)} { \lvert  Q\rvert ^{d/n} } 
\int _{\mathbb R ^{n}} \frac {\lvert  Q\rvert ^{d/n} } 
{ \lvert  Q\rvert ^{2d/n} +\textup{dist}(x, Q) ^{2d/n} } \; \sigma (dx) \leq  \mathscr A _2 
\\
\int _{Q} \lvert \mathsf  R _{\sigma } \mathbf 1_{Q} (x)\rvert ^2 \; w(dx) \le \mathscr T ^2 \sigma (Q),
\end{gather*}
where $ \mathsf R _{\sigma }$ is the Riesz transform as above, and 
the dual conditions are obtained by interchanging the roles of the two weights.
Examples show that a key step of the proof fails in absence of the extra geometric condition imposed on the 
weights.  
\end{abstract}

\maketitle
 
\setcounter{tocdepth}{1}
\tableofcontents

\section{Introduction} 
We are interested in the two weight inequality for  fractional Riesz transforms.  Namely, for two weights, non-negative Borel measures $ w $ and $ \sigma $ 
on $ \mathbb R ^{n}$, 
we consider the  norm inequality  
\begin{equation}\label{e:N}
\left\lVert 
\int_{\mathbb{R}^n}  f (y) \frac {x-y} {\lvert  x-y\rvert ^{d+1} } \; \sigma (dy)
\right\rVert_{L ^{2} (\mathbb{R}^n;w)} \le \mathscr N \lVert f\rVert_{L ^2 (\mathbb{R}^n;\sigma )}. 
\end{equation}
Here, $ 0< d\neq n-1 \le n$,  and the kernel above is vectorial, making it the standard $ d$-dimensional Riesz transform on $ \mathbb R ^{n}$.  
Throughout, we take $ \mathscr N$ to be the best constant in the inequality above.
To avoid cumbersome notation, we shorten the inequality above to 
\begin{equation*}
\lVert R _{\sigma } f \rVert_{ w } \le \mathscr N \lVert f\rVert_{\sigma } . 
\end{equation*}
We also suppress the dimensionality of the Riesz transform,  and 
since our results only hold in $ L ^2 $, we write $ \lVert f\rVert_{\sigma }\equiv  \lVert f\rVert_{L ^2 (\mathbb R ^{n}; \sigma )}$.  
The norm inequality is written with the weight $ \sigma $ on both sides of the inequality
since it then dualizes by interchanging the roles of the weights.

The desired  characterization of the boundedness of the Riesz transform is in terms of a pair of testing inequalities and an $ A_2$ type condition phrased in the language of this Poisson-like operator. For a cube $ Q\subset \mathbb R ^{n}$,  set 
\begin{equation}\label{e:Pr}
\mathsf P ^{\textup{r}} (\sigma ,Q) \equiv 
\int _{\mathbb R ^{n}} \frac {\lvert  Q\rvert ^{d/n} } 
{ \lvert  Q\rvert ^{2d/n} +\textup{dist}(x, Q) ^{2d} } \; \sigma (dx). 
\end{equation}
The superscript $ {} ^{\textup{r}}$ denotes `reproducing   Poisson' due to the similarities with the analogous power in the setting of weighted estimates on reproducing kernel Hilbert spaces.

Below, we impose the condition that the dimension of the Riesz transform \emph{not} be of codimension one.  
But, we also impose an extra geometric condition on each weight individually.  Say that the weight $ w$ is \emph{uniformly of full dimension} 
if there is a $ 0<\eta <1$ such that for all cubes $ Q \subset \mathbb R ^{n}$, 
\begin{equation}\label{e:ufd}
\inf_H \int _{Q}\int _{Q} \frac { \textup{dist}(x', H+x)  } { \lvert  x-x'\rvert  }  \; w (dx)\, w (dx') \ge \eta w (Q) ^2 .  
\end{equation}
The infimum is formed over all hyperplanes $ H$ through the origin of co-dimension one.  
Without this assumption, a crucial step in the proof of   Lemma~\ref{l:E}  fails. 
We thank Xavier Tolsa for pointing this out to us; elaborations of his example are given in \S\ref{s:example}.  

A weight $ w$ is doubling if there is a constant $ C _{d}$ so that for all cubes $ Q$, $ w (2Q) \le C_d w (Q)$, where $2Q$ is the cube concentric with $Q$ but twice the side length. 
Such weights are uniformly of full dimension, so that in the case of both $ \sigma $ and $ w$ doubling, and 
$ n=d \ge 2$, the Theorem below provides a characterization of the two weight inequality for the Riesz transforms.

\begin{theorem}\label{t:main} Assume that $ 0< d \neq n-1 \leq n$ and that $ \sigma $ and $ w$ are two weights on $ \mathbb R ^{n}$,  
each of which are uniformly of full dimension. 
The norm inequality \eqref{e:N} holds if and only if 
for finite positive constants $ \mathscr A_2, \mathscr T$, these inequalities hold uniformly 
over cubes $ Q$, and over their  dual formulations 
\begin{gather}\label{e:A2}
\frac { w (Q)} { \lvert  Q\rvert ^{d/n} }  \cdot  \mathsf P ^{\textup{r}} (\sigma ,Q)
\leq \mathscr  A_2,
\\  \label{e:test}
\int _{Q} \lvert  R _{\sigma } \mathbf 1_{Q} (x)\rvert ^2 \; w(dx) \le \mathscr T ^2 \sigma (Q), 
\end{gather}
(The dual statements are obtained by interchanging the roles of $ w$ and $ \sigma $.) 
Moreover, 
$\mathscr N \simeq \mathscr R\equiv \mathscr A_2 ^{1/2} + \mathscr T$.  
\end{theorem}

The Theorem has two critical side conditions. The first is the uniformly full dimension assumption, 
which we need to derive a critical energy inequality.  The second, of no common point masses, 
we can eliminate provided $ n-1 < d \leq n$.  We assume it in the general case in order to 
avoid some complications in the functional energy argument.  A precise statement 
of a more general result is given in Theorem~\ref{t:exact}.

The difficult part of the theorem is to show the sufficiency of the $ A_2$ and testing inequalities.  
The key property of Riesz transforms that we will exploit is the divergence condition 
\begin{equation} \label{e:div}
\operatorname {div} \frac {y} {\lvert  y\rvert ^{d+1} } = \frac {n-d-1} {\lvert  y\rvert ^{d+1} } \neq 0. 
\end{equation}
Note that in the case that the codimension is not one, i.e. $d+1\neq n$, the divergence of the kernel is signed, while in the case of codimension one, i.e. $d+1=n$, we gain no information since the divergence vanishes.  

Concerning the additional condition that the weights be of \emph{uniformly full dimension}, we have 

\begin{proposition}\label{p:ufd} Sufficient conditions for a weight $ w$ to be of uniformly full dimension are either 
(1) the weight is doubling, or (2) the weight is Ahlfors-David regular of parameter $ n-1 < d \leq n$.  
The latter condition means that there are constants $ 0< C_0 < C_1 < \infty $ so that  for all points $ x $ in the support of $ w$, 
and balls $ B (x,r)$ centered at $ x$, of radius $ 0< r < \infty $, there holds 
\begin{equation} \label{e:AD}
C_0 r ^{d} < w (B (x,r)) < C_1 r ^{d}.  
\end{equation}

\end{proposition}

\smallskip 

We adopt key elements of the proof from the Hilbert transform \cites{MR3285857,MR3285858,13120843}, with essential modifications.  
(1) An energy inequality, derived from the sufficient conditions for 
norm boundedness, is a decisive tool.
This occurs in a very clean way for the Hilbert transform, but has additional complications 
herein, complications that are close to those of the Cauchy transform \cite{combined}.   
The divergence property and the side condition of being uniformly of full dimension are essential here.  See Lemma~\ref{l:E}.  
The divergence condition says that the trace of the tensor $ \nabla \mathsf R _{\sigma } f $ is not zero, namely 
there is at least one non-negative eigenvector.  But, by the example of Proposition~\ref{p:example}, the Riesz transform 
can be nearly zero on a set as large as a square of co-dimension one. But, the weights cannot be concentrated on 
such sets by the uniformly full dimension assumption. 
(2) The monotonicity inequality dominates  certain off-diagonal terms.
Estimates of this type are standard, but in the two weight setting they 
have an important modification to be consistent with the energy inequality.  
The inequality \eqref{e:mono}  now has two components, 
arising from first and second gradient calculations. They are, by example, incomparable terms. 
Hence,  both play a role in the off-diagonal considerations which dominate the proof. 
(3) In the global-to-local reduction, we use the  argument of  Hyt\"onen \cite{13120843}. 
A certain complication arises, in the case of $ 0< d < n-1$, and in order to give a short argument, 
we impose the $ A_2$ condition `without holes.' 
(4) The local part of the proof, an argument invented in \cite{MR3285858}, has additional complications that are not present in other proofs of this type. 
(5) The $ A_2$ condition, with its large exponent, if $ d> 1$, requires new arguments to apply. 
Curiously, it is only needed in a couple of `standard' lemmas, see Lemma~\ref{l:nearby}.  
 (6) The proofs for the Hilbert transform, likewise for the Cauchy operator \cite{combined}, employ Muckenhoupt's two weight inequality for the Hardy inequality.  This is of course no longer available to us, creating a new complication. 
We use the deep technique of \emph{surgery}, invented by Nazarov-Treil-Volberg \cite{MR1998349}, 
and a fundamental tool in local $ Tb$ theorems.  In our setting, the application in Lemma~\ref{l:surgery} is new, but not  very difficult.

\section{Conventions} 

\begin{itemize}
\item Letters like $ P,Q,K, L$ will denote a cube, and it is convenient to 
denote the \emph{scale} of cube $Q $ by $ \ell (Q) \equiv \lvert  Q\rvert ^{1/n} $.  
The geometric center of the cube $ Q $ is denoted by $ x_Q$.    

\item We  will introduce two distinct grids $ \mathcal D _{\sigma }$ and $ \mathcal D _{w}$, 
associated with the corresponding $ L ^2 $ spaces.  We let $\pi Q$ denote the parent of the cube $Q$ in the grid $\mathcal{D}_{\sigma}$ (or $\mathcal{D}_{w}$, which should be clear from context).  We will construct subgrids, for instance 
$ \mathcal F \subset \mathcal D _{\sigma }$. In this case, we set $ \pi _{\mathcal F} Q $ 
to be the smallest element of $ \mathcal F$ which contains $ Q$. This is well-defined for 
all $ Q\in \mathcal D _{\sigma }$ which are contained in some element of $ \mathcal F$.  
We then set $ \pi _{\mathcal F} ^{1} Q = \pi _{\mathcal F} Q$, and inductively set 
$ \pi _{\mathcal F} ^{t+1} Q$ to be the member of $ \mathcal F$ that \emph{strictly} contains $ \pi ^{t} _{\mathcal F} Q$.

\item  For two cubes $ P, Q$, and integer $ s$,  we will write $ Q\Subset_s P_Q$, if $ Q\subset P$ and $ 2 ^{s} \ell (Q) \le \ell (P) $. 
This notation will be used for $ s=r$ and $ s=4r$, where $ r$ will be a fixed, large, integer, associated with the 
notion of goodness (defined below). This notation will be used when the cubes are in the same, or two distinct grids.  
(In the case of distinct grids, we will in addition be able to assume that $ 2 ^{r} \ell (Q)\le \ell (P)$ and $ Q\cap P\neq \emptyset $ implies 
	$ Q\subset P$.) 

\item   The inclusion $ \Subset _r $ will also be used to define parents, with $  \dot\pi _{\mathcal F} Q= F$,  meaning that $ F \in \mathcal F$ is the minimal element such that  $ Q \Subset _{r} F$.  

\item We will very frequently apply operators to indicators of cubes, and so will identify $ \mathbf 1_{Q} \equiv Q $ at many points in the proof.

\item 
Averages appear repeatedly in the analysis below.  The average of $ f$ over cube a $ Q$, with respect to the weight $ \sigma $ is 
\begin{equation*}
[ f ] ^{\sigma } _{Q} \equiv \frac 1 {\sigma (Q)} \int _{Q} f(x) \; \sigma (dx), 
\end{equation*}
this provided $ \sigma (Q) >0$.  
In particular, when we write
\begin{equation*}
[x]^{\sigma}_{Q}\equiv  (   [x_1 ] ^{\sigma } _Q , \ldots, [x_n]^\sigma_Q ) 
\end{equation*}
we mean the $\sigma$ center of mass of the cube $Q$.

\item The three Poisson averages that are relevant are 
\begin{gather}
\textup{(Reproducing)}  \qquad 
\mathsf P ^{\textup{r}} (\sigma ,Q) \equiv 
\int _{\mathbb R ^{n}} \frac {\lvert  Q\rvert ^{d/n} } 
{ \lvert  Q\rvert ^{2d/n} +\textup{dist}(x, Q) ^{2d/n} } \; \sigma (dx) , 
\\   \label{e:Cgrad}
\textup{(Gradient)}  \qquad 
\mathsf P_{\sigma} ^{\textup{g}} ( f ,Q) \equiv 
\int _{\mathbb R ^{n}}  \lvert  f(x)\rvert  \frac {\ell (Q)  } 
{ \ell (Q) ^{d+1} +\textup{dist}(x, Q) ^{d+1} } \; \sigma (dx), 
\\  \label{e:Cgrad+}
\textup{(Gradient Plus)}  \qquad  
\mathsf P ^{\textup{g}+}_{\sigma } (  f, Q ) \equiv 
 \int _{\mathbb R ^{n}}  \lvert  f(x)\rvert \frac { \ell (Q) ^{2} } 
 { \ell (Q) ^{d+2} + \textup{dist}(x,Q) ^{d+2}} \; \sigma (dx). 
\end{gather}
The last two Poisson averages appear in the monotonicity estimate \eqref{e:mono}, but in somewhat different roles. 
The more complicated role of the gradient plus Poisson operator there is (strongly) mitigated by the estimate \eqref{e:g+<g}.  

\item Energy is defined by $ E (w, K) ^2 =\frac 2 {w (K)} 
\sum_{Q \in \mathcal D_ \sigma  \;:\; Q\subset K} \left\lVert  \Delta ^{w} _{Q} \frac x { \ell (K)} \right\rVert_{w} ^2 $.  
The key assumption is the finiteness of the energy constant given in \eqref{e:E}. 

\item Using energy, we will use this partition of  a cube $ Q \in \mathcal D _{\sigma }$.  Set $ \mathcal W_Q$ to be the maximal cubes $ K\in \mathcal D _{w}$  (or $ \mathcal D_ \sigma  $) such that 
$ 2 ^{r} \ell (K) \leq \ell (Q) $, and $ \textup{dist} (K, \partial Q') \ge \ell (K) ^{\epsilon } \ell (Q') ^{1- \epsilon } $ 
for all dyadic cubes $ Q'\subset Q$, with $ 2 ^{r} \ell (K) \le \ell (Q')$.   See Definition~\ref{d:W} and their Whitney property in Proposition~\ref{p:w}.

\item Constants.  We want to estimate $ \mathscr N$, the norm of the Riesz transform in terms of the $ A_2$ constant, the energy constant, and the testing constant.  Namely, that $ \mathscr N \lesssim \mathscr R \equiv \mathscr A_2 ^{1/2} + \mathscr E + \mathscr T$.

 \end{itemize}

\bigskip 

\subsection{Dyadic Grids}

Let $\hat{\mathcal{D}}$ denote the standard dyadic grid in $\mathbb R $.  A \textit{random} dyadic grid $\mathcal{D} ^{r}$ is specified by $\xi\in\{0,1\}^{\mathbb Z }$ 
and choice of $ 1\le \lambda \le 2$.  The elements of $ \mathcal D^{r}$ are  given by
$$
I\equiv \hat{I}\dot+\xi = \lambda \biggl\{ \hat{I}+\sum_{n: 2^{-n}<|\hat{I}|} 2^{-n}\xi_n \biggr\}. 
$$
Place the uniform probability measure $\mathbb{P}$ on $ \xi \in \{0,1\}^{\mathbb Z }$, and choose $ \lambda $ with respect to normalized measure 
on $ [1,2]$ with measure $ \frac {d \lambda } \lambda $.

Let $ \mathcal D _{\sigma }$ be a $ n$-fold tensor product of independent copies $ \mathcal D ^{r}_1 \times \cdots \times \mathcal D ^{r}_n$, 
which is used in $ L ^2 (\mathbb{R}^n;\sigma )$.  
It is imperative to use a second random grid $ \mathcal D _{w}$ for $ L ^2 (\mathbb{R}^n;w)$, but the subscripts on the two random grids are 
frequently suppressed, and largely irrelevant to the argument, except for the surgery argument in Lemma~\ref{l:surgery}.

A choice of grid $ \mathcal D _{\sigma }$ is said to be \emph{admissible} if $ \sigma $ does not assign positive mass 
to any lower dimensional face of a cube $ Q\in \mathcal D _{\sigma }$.  By the construction of the random grid, in 
particular the use of dilations, a randomly selected grid is admissible with probability one.  This is 
always assumed below, for both $ \mathcal D _{\sigma }$ and $ \mathcal D _{w}$. We assume that the dilation 
parameter is $ \lambda =1$, which should not cause any confusion.

For $Q\in \mathcal D _{\sigma }$, we speak of  $ \textup{ch} (Q)$, the \emph{children} of $ Q$, the $ 2 ^{n}$ maximal elements of $ \mathcal D$ which are strictly 
contained in $ Q$. If $ \sigma (Q) >0$ and $ \sigma (Q') >0$ for at least two children of $ Q$, then we define the martingale difference 
\begin{equation*}
\Delta ^{\sigma } _{Q} f \equiv   - [ f ]  ^{\sigma }_{Q} \cdot Q + \sum_{Q'\in \textup{ch} (Q) \::\: \sigma (Q') >0 } [ f ] ^{\sigma } _{Q'} \cdot Q'. 
\end{equation*}
Otherwise, $ \Delta ^{\sigma } _{Q} f \equiv 0$. In this definition, we are identifying the cube $ Q$ with $ \mathbf 1_{Q}$, which we will do throughout.  
It is well-known that  $ f = \sum_{Q\in \mathcal D_\sigma} \Delta ^{\sigma } _{Q} f$, and that 
\begin{equation}
\lVert f\rVert_{\sigma } ^2 = \sum_{Q\in \mathcal D_\sigma} \left\lVert \Delta ^{\sigma } _{Q} f \right\rVert_{\sigma } ^2 . 
\end{equation}

\subsection{Good and Bad Cubes}

Fix $0<\epsilon<1$ and $r\in\mathbb{N}$.  An cube $I\in\mathcal{D}_\sigma$ is said to be $(\epsilon,r)$-\textit{bad} if there is an cube $J\in\mathcal{D}$ such that $\ell (J)>2^{r}\ell(I)$ and $\textnormal{dist}(I,\partial J)<\ell(I)^{\epsilon}\ell(J)^{1-\epsilon}$.  Otherwise, an cube $I$ will be called $(\epsilon,r)$-\textit{good}.  We have the following well-known properties associated to the random dyadic grid $\mathcal{D}_\sigma$.

\begin{proposition}
\label{RandomLattice}
The following properties hold:
\begin{enumerate}
\item The property of $I=\hat{I}\dot +\xi$ being $(\epsilon,r)$-good depends only on $\xi$ and $\vert I\vert$;
\item $\mathbf{p}_{\textnormal{good}}\equiv\mathbb{P}\left(I \textnormal{ is } (\epsilon,r)-\textnormal{good}\right)$ is independent of $I$;
\item $\mathbf{p}_{\textnormal{bad}}\equiv 1-\mathbf{p}_{\textnormal{good}}\lesssim \epsilon^{-1} 2^{-\epsilon r}$.
\end{enumerate}
\end{proposition}
A similar result of course holds for the grid $\mathcal{D}_w$.  Now, write the identity operator in $L^2(\mathbb{R}^n;\sigma)$ as
$$
f=P^\sigma_{\textnormal{good}} f+P^\sigma_{\textnormal{bad}}f \quad\textnormal{ where } P^\sigma_{\textnormal{good}} f\equiv \sum_{Q\in\mathcal{D}: Q \textnormal{ is } (\epsilon, r)-\textnormal{good}} \Delta_I^{\sigma} f.
$$
Similar notation applies for the identity operator on $L^2(\mathbb R^n;w)$.  Below, we will frequently impose the condition that the cubes are good, and will not explicitly point this out in the notation.

We have the following well-known proposition in this context.
\begin{proposition}  
\label{BadExpectation}
The following estimate holds:
$$
\mathbb{E}\left\Vert P_{\textnormal{bad}}^\sigma f\right\Vert_{\sigma}^2\lesssim \epsilon^{-1} 2^{-\epsilon r}\Vert f\Vert_{\sigma}^2.
$$
An identical estimate is true for the weight $w$.
\end{proposition}

\section{Energy Inequalities} 

We begin the discussion of energy inequalities with the case of codimension other than one, $ 0< d\neq n-1\leq n$.  

Crucial to this discussion is the introduction of two different  Poisson averages, 
with critically, a different power than  that of the `reproducing Poisson' average in  \eqref{e:Pr}. 
The Poisson-like average arise  from gradient considerations, 
hence we write a superscript $ {} ^{\textup{g}}$ on it, and $ {} ^{\textup{g}+}$ on the second.  
\begin{align}\label{e:Pg}
\mathsf P_{\sigma} ^{\textup{g}} ( f ,Q) \equiv 
\int _{\mathbb R ^{n}}  \lvert  f(x)\rvert  \frac {\ell (Q)  } 
{ \ell (Q) ^{d+1} +\textup{dist}(x, Q) ^{d+1} } \; \sigma (dx), 
\\ \label{e:Pg+}
 \mathsf P ^{\textup{g}+}_{\sigma } (  f, Q ) \equiv 
 \int _{\mathbb R ^{n}}  \lvert  f(x)\rvert \frac { \ell (Q) ^{2} } 
 { \ell (Q) ^{d+2} + \textup{dist}(x,Q) ^{d+2}} \; \sigma (dx). 
\end{align}
We are making the operators be sublinear, inserting the absolute value of $ f$ on the right, since we only ever need to use them for non-negative functions. 
It is important to note that the reproducing Poisson decay is $ 2d$, as in \eqref{e:Pr}, whereas the 
decay for $ \mathsf P_{\sigma}^{\textup{g}}$  is 
$ d+1$. These agree when $ d=1$,  e.g.~the case of the Hilbert transform,  which is included in our 
discussion, and  the Cauchy transform, which is not.  More generally, the 
reproducing decay is slower for $ 0< d <1$, but is otherwise faster.
Faster decay on the reproducing  kernel creates additional technical problems for us.  
Fortunately, we will \emph{however} not find it necessary to distinguish these three cases in the analysis below.

The \emph{energy} of $ w$ over the cube $ K \in \mathcal D_ \sigma $ is taken to be 
\begin{align}\label{e:energy1}
E(w,K) ^2  & \equiv \frac 1 {w (K) ^2 }
\int _{K} \int _{K}
\frac { \lvert  x-x'\rvert ^2  } { \ell (K) ^2    } \; w(dx)\, w(dx')
\\
& = \frac 2 {w (K)} 
\int _{K} \frac { \lvert  x- [x] _{K} ^{w} \rvert ^2  } { \ell (K) ^2    } \;  w(dx)
\\ \label{e:xe}
&= \frac 2 {w (K)} 
\sum_{Q \in \mathcal D_ \sigma  \;:\; Q\subset K} \left\lVert  \Delta ^{w} _{Q} \frac x { \ell (K)} \right\rVert_{w} ^2 .
\end{align}
The energy is the norm of $ \frac{x}{\ell (K)} \cdot K $ in  $ L ^2 _0 (\mathbb{R}^n;w)$, the subspace in $L^2(\mathbb{R}^n;w)$ that is orthogonal to constants. 
In the middle line, we are subtracting off the mean value of $ x \cdot K$, in particular, 
$ E (w,K) \le 1$, and is as small as zero if $ w \cdot Q$ is just a point mass. 
It is easy to check the equality of the three expressions above, and we will use all three.  
And, one should be careful to note that the last equality requires that $ K\in \mathcal D _{\sigma }$, a condition 
we will not always have. 
Observe that  
\begin{equation*}
\sum_{Q \in \mathcal D_w \;:\; Q\subset K} \Bigl\lVert  \Delta ^{w} _{Q} \frac x { \ell (K)} \Bigr\rVert_{w} ^2 
\leq E(w,K) ^2\,w(K).  
\end{equation*}
The difference between this last display and \eqref{e:xe} is that 
here is that we are summing over cubes in $ \mathcal D _{w}$, while in \eqref{e:xe}, we are summing over cubes in $ \mathcal D _{\sigma }$.

The energy constant is defined in terms of supplemental constants $ 0< C_0, C_1$, both of which 
are functions of $ n$ and $ d$. We will comment in more detail on the selection 
of these constants below.  The \emph{energy constant} is the 
 best constant $ \mathscr E = \mathscr E (C_0, C_1)$ in the inequality 
\begin{equation}\label{e:E}
\sum_{K\in \mathcal K}  \mathsf P^{\textup{g}}_ \sigma  (Q_0 \setminus C_0K , K) ^2  
 E (w,K) ^2  w (K)
\le \mathscr E ^2   \sigma (Q_0).  
\end{equation}
Here, $ Q_0 \subset \mathbb R ^{n} $ is a cube, and $ \mathcal K \subset \mathcal D _{\sigma }$ or $ \mathcal K \subset \mathcal D _{w}$ 
is any partition of $ Q_0$ into dyadic cubes for  which this Whitney type condition holds,  
\begin{equation}\label{e:bounded}
\sum_{K\in \mathcal K} (C_0 K) (x) \le C_1 Q_{0}(x), \qquad x\in \mathbb R ^{n}.    
\end{equation}
Importantly, we need the inequality \eqref{e:E} above with the roles of the weights reversed. 
And $ \mathscr E$ will denote the best constant in \eqref{e:E} and its dual statement. 

With this notation the precise result we are proving in this paper is as follows. 
In particular, note that the condition that the pair of weights have no common point mass, and a stronger $ A_2$ condition, 
are imposed in the case of $ 0< d\leq n-1$. 

\begin{theorem}\label{t:exact} Let $ \sigma $ and $ w$ be two weights on $ \mathbb R ^{n}$, and let $ 0< d  \leq n$. 
Assume that the energy constant $ \mathscr E$ as defined in \eqref{e:E} is finite.  
\begin{enumerate}
\item For $ n-1 < d \leq n$, assume the $ A_2$ condition `with holes', namely 
\begin{equation*}
\sup _{ \textup{$ Q$ a cube }} \frac  {\sigma (Q) } {\ell (Q) ^{d}} \mathsf P _{w}  ^{r}( \mathbb R ^{n} \setminus Q, Q)  
+ \frac  {w (Q) } {\ell (Q) ^{d}} \mathsf P _{\sigma }  ^{r}( \mathbb R ^{n} \setminus Q, Q)  = \mathscr A_2 < \infty . 
\end{equation*}
Here, we are using the reproducing kernel Poisson average, as defined in \eqref{e:Pr}.  

\item For $ 0< d \leq  n-1$, assume that $ \sigma $ and $ w$ do not share a common point mass, and that the the $ A_2 $ 
condition `with no holes' below holds.  
\begin{equation*}
\sup _{ \textup{$ Q$ a cube }} \frac  {\sigma (Q) } {\ell (Q) ^{d}} \mathsf P _{w}  ^{r}( \mathbb R ^{n}, Q)  
+ \frac  {w (Q) } {\ell (Q) ^{d}} \mathsf P _{\sigma }  ^{r}( \mathbb R ^{n}, Q)  = \mathscr A_2 < \infty . 
\end{equation*}

\end{enumerate}
Finally, assume that the two testing inequalities \eqref{e:test} and their duals hold.  Then, the two weight norm inequality 
\eqref{e:N} holds, and $ \mathscr N \lesssim \mathscr A_2 ^{1/2} + \mathscr E + \mathscr T  \equiv \mathscr R$.  

\end{theorem}

Sawyer, Uriarte-Tuero, and Shen \cite{13104484} have formulated a result along these lines.  Moreover, their paper \cite{MR3431617} gives example weights which show that the energy condition need not follow from the other hypotheses, which it does in the case of the Hilbert transform.  

This is the one place in which the side condition of $ \sigma $ and $ w$ being uniformly of full dimension is used.  
\begin{lemma}[Uniformly Full Dimension implies Energy] 
\label{l:E}   Assume   $ 0< d \neq n-1 < n$ and  both $ \sigma $ and $ w$ are uniformly of full dimension, in the sense of \eqref{e:ufd}, with 
constant $ 0< \eta <1$. In addition assume that they don't share a common point mass.  
There is a constant $ C_0 = C_0 (n,d, \eta )$, absolute, so that this holds.   
Let $ \sigma $ and $ w$ be a pair of weight for which the $ A_2$ hypothesis \eqref{e:A2} 
and the testing inequalities \eqref{e:test} and the dual to \eqref{e:test} holds.  
Then $ \mathscr E ^2 \lesssim  \mathscr A_2 + \mathscr T ^2 $.  The implied constant depends upon $ C_1$, and the constants $ \eta $ that enter into \eqref{e:ufd}.
\end{lemma}

\begin{proof}
The expression on the left in \eqref{e:E} 
is a sum of positive terms and so we can assume that the sum above is over just a finite number of terms.  
This is the main inequality: 
\begin{align} \label{e:xx}
[\mathsf P^{\textup{g}} (\sigma \cdot Q_0\setminus  C_0K , K)] ^2 E (w,K) ^2  w (K) 
\lesssim \int _{K} \lvert \mathsf R _{\sigma } (Q_0 \setminus C_0K)  (x)\rvert ^2 \;  w(dx) .   
\end{align}
Using linearity in the argument of the Riesz transform, and the testing inequality \eqref{e:test}, one sees that
\begin{align*}
 \sum_{K\in \mathcal K} \int _{K} 
\lvert \mathsf R _{\sigma } Q_0   \rvert ^2 \; d w &\le \mathscr T ^2 \sigma (Q_0) , 
\\
 \sum_{K\in \mathcal K}  \int _{K} 
\lvert \mathsf R _{\sigma } C_0K  \rvert ^2 \; d w &\le \mathscr T ^2  
 \sum_{K\in \mathcal K} \sigma (C_0K) \le C_1\mathscr T ^2 \sigma (Q_0). 
\end{align*}
Note that the inequality \eqref{e:bounded} is used here in the second inequality, while for the first we use that $\mathcal K$ is a partition of $Q_0$.  These two estimates coupled with \eqref{e:xx} prove \eqref{e:E}, which would give the statement of the Lemma.

\medskip 
So, it remains to prove \eqref{e:xx}. Assume that it fails, namely for a constant $ 0<\kappa < 1 $ that we will pick, as a function of $ 0<\eta<1 $ 
in \eqref{e:ufd}, that there holds 
\begin{equation*}
\int _{K} \lvert \mathsf R _{\sigma } (Q_0 \setminus C_0K)  (x)\rvert ^2 \;  w(dx) \leq 
\kappa [\mathsf P^{\textup{g}} (\sigma \cdot Q_0\setminus  C_0K , K)] ^2 E (w,K) ^2  w (K) . 
\end{equation*}
We do not know how to use this condition directly, passing instead to it's implication that 
\begin{equation}\label{e:kappa}
\begin{split}	
 \int _{K}\int _{K} \lvert \mathsf R _{\sigma } (Q_0 \setminus C_0K)  (x') &-\mathsf R _{\sigma } (Q_0 \setminus C_0K)  (x)\rvert ^2 \;  w(dx')\, w (dx)
\\&
\leq 2\kappa [\mathsf P^{\textup{g}} (\sigma \cdot Q_0\setminus  C_0K , K)] ^2 E (w,K) ^2  w (K)^{2 }  .
\end{split}
\end{equation}
Fix $ x$ as in the integral on the left, and consider the symmetric tensor $ T_x = \nabla \mathsf R _{\sigma } (Q_0 \setminus C_0 K) (x)$, 
which by the Spectral Theorem,  has a diagonalization.  
By the divergence equality \eqref{e:div}, the trace of this tensor is, up to a sign 
\begin{equation*}
P _{x} \equiv \int _{Q_0 \setminus C_0 K} \frac {\lvert n-d-1 \rvert} { \lvert  x-y\rvert ^{d+1}} \;  \sigma(dy) . 
 \end{equation*}
Thus, there is at least one eigenvalue of $ T_x$ of magnitude $ c \cdot P _{x}$, and the maximal eigenvalue is $ C \cdot P _{x}$. 

Observe that $ P _{x}$ and $ T _{x}$ are  essentially constant on $ K$: For $ x',x\in K$: 
\begin{equation} \label{e:ksmall}
\lvert  P _{x'} - P _{x}\rvert + \lvert  T _{x'}-T _{x}\rvert \lesssim  \frac { \lvert  x'-x\rvert } { C_0 \ell (K)}  P _{x_K}.  
\end{equation}
This depends only a second derivative calculation.  
We will choose $ C_0 \gtrsim \eta  ^{-4}$, and $ \kappa \simeq \eta ^{4}$, where $ \eta $ is as in \eqref{e:ufd}.  
In particular, the right hand side above will be quite small. 

Define 
\begin{equation*}
L _{x} = \{  y \in \mathbb R ^{n} \::\: \lvert  y\rvert= 1,\   \lvert   T _x  y\rvert < \sqrt {\kappa}  P_x    \}. 
\end{equation*} 
By our diagonalization observation, there must be a  a hyperplane $ H_x$ with 
\begin{equation*}
\sup _{ y\in L _{x}} \textup{dist}( y, H_x) \lesssim \sqrt {\kappa} . 
\end{equation*}
Take $ H_x$ to be orthogonal to the  eigenvector with maximal eigenvalue.   By \eqref{e:ksmall}, we can in addition take $ H _x=H$, 
namely independent of $ x\in K$.  

  For $ x'$ as in \eqref{e:kappa}, set $ v $ to be the unit vector in the direction $ x'-x$ 
and set $ \delta = \lvert  x'-x\rvert $.  Then, 
\begin{equation*}
\mathsf R _{v, \sigma } (Q_0 \setminus C_0 K) (x') -\mathsf R _{v,\sigma } (Q_0 \setminus C_0 K) (x) 
=  \delta \cdot   T _x v + O (C_0 ^{-1}  P_x). 
\end{equation*}
Indeed,  the first term on the right is the first derivative approximation to the difference. 
By Taylor's theorem, we should have a second derivative error term, but this is  controlled by \eqref{e:ksmall}.

The inequality \eqref{e:kappa} can be rewritten as 
\begin{align*}\frac 1 {w (K)}
\int _{K} \int _{K}  \delta ^2  \lvert   T _x v \rvert ^2 \; dw \, dw 
\leq C \kappa \frac 1 {w (K)} \int _{K}  \int _{K}  \delta ^2  \cdot P_x ^2  \; dw \, dw 
\end{align*}
where we are suppressing the dependence of $ v$ and $ \delta $ on $ x $ and $ x'$,  
therefore, 
\begin{equation*}
\int _{K} \int _{K}   \frac { \textup{dist}(x' , H +x) ^2 } { \lvert  x-x'\rvert ^2  } \; w (dx)\, w (dx')  \lesssim \sqrt \kappa  w (K ) ^2 .  
\end{equation*}
 This follows since the left hand side we recognize the difference between the Riesz transforms as the symmetric tensor $T_x$ and the right hand side we use the computation of the trace of the tensor $T_x$ and estimating it by its trace $P_x$. 
 The difference between this and \eqref{e:ufd} is   the presence of the square above. But due to the normalization by $ w (K)$, we have 
\begin{equation*}
\int _{K} \int _{K}   \frac { \textup{dist}(x' , H +x)  } { \lvert  x-x'\rvert  } \; w (dx)\, w (dx')  \lesssim   \kappa ^{1/4}  w (K ) ^2.  
\end{equation*}
Assuming  the failure of \eqref{e:xx} with $ \kappa \simeq \eta ^{4}$ leads to a contradiction of the assumption \eqref{e:ufd}.  The proof is complete.

\end{proof}

The partitions that we use  with the term energy are defined here. 
\begin{definition}\label{d:W}For a cube $ Q \in \mathcal D _{\sigma }$, set $ \mathcal W_Q$ to be the maximal cubes $ K\in \mathcal D _{w}$  (or $ \mathcal D_ \sigma  $) such that 
$ 2 ^{r} \ell (K) \leq \ell (Q) $, and $ \textup{dist} (K, \partial Q') \ge \ell (K) ^{\epsilon } \ell (Q') ^{1- \epsilon } $, for $ Q'\supset Q$. 
\end{definition}
These are the maximal cubes in the dual dyadic grid that are `good with respect to all cubes larger than $ Q$.' 
They have this Whitney property. 

\begin{proposition}\label{p:w}  For any finite $ C_0$, and $ 2 ^{r (1- \epsilon )} > 4 C_0$, there holds 
\begin{equation*}
\sum_{K\in \mathcal W_Q} (C_0 K) (x) \le  (2r C_0)  Q (x). 
\end{equation*}
\end{proposition}

\begin{proof}
The condition on $ C_0$ and $ r$, and the selection criteria for $ \mathcal W_Q$ imply that $ C_0 K\subset Q$ for all $ K\in \mathcal W_Q$. 
So we need only control the overlaps.  
Suppose there are $ K_1 , K_2 \in \mathcal W_Q$ with $ 2 ^{r}\ell (K_1) \le  \ell (K_2)$, and $ C_0 K_1 \cap C_0 K_2 \neq \emptyset $.  
Then, we have 
\begin{equation*}
\textup{dist}(K_1 , \partial Q) \ge \ell (K_2) ^{\epsilon } \ell (Q) ^{1- \epsilon }  - C_0 \ell (K_2)
=  (1 - C_0 2 ^{-r (1- \epsilon )}) \ell (K_2) ^{\epsilon } \ell (Q) ^{1- \epsilon } . 
 \end{equation*}
 That is, the cube $ K_1$ with small side length is rather far from the boundary of $ Q$.  
 But, by maximality of $ K_1$, we must have some $ Q'\supset Q$ with 
 \begin{equation*}
\textup{dist}(K_1 , \partial Q')  \leq 2 ^{\epsilon } (\ell K_1) ^{\epsilon } \ell (Q') ^{1- \epsilon }. 
\end{equation*}
And then, turning to the larger cube $ K_2$, we see that 
\begin{align*}
 (\ell K_2) ^{\epsilon } \ell (Q') ^{1- \epsilon } & \leq \textup{dist}(K_2 , \partial Q') 
 \\
 & \leq 2 ^{\epsilon } (\ell K_1) ^{\epsilon } \ell (Q') ^{1- \epsilon }. 
\end{align*}
This contradicts the selection criteria for $ K_2$.   
\end{proof}

\subsection{Sufficient Conditions for Uniformly Full Dimension}

We give the proof of Proposition~\ref{p:ufd}.  
First, if $ w$ is \emph{doubling}, then there is a constant $C_d $ so that for all cubes $ Q\subset \mathbb R ^{n}$,  
there holds $ w (2Q) \leq C_d w (Q)$.  To check the condition \eqref{e:ufd}, take a cube $ Q$,  
and partition it into $ \mathcal P$, cubes of side length $ \frac 14 \ell (Q)$.  For each $ x\in Q$, and 
hyperplane $ H$ of co-dimension one, we can choose $ Q', Q'' \in \mathcal P$ so that  $ x\in Q''$ and for any $ x'\in Q'$, the following holds 
\begin{equation*}
\textup{dist} (x', H+x) \simeq \textup{dist}(Q',Q'') \simeq \ell (Q).  
\end{equation*}
Since for each $ Q'\in \mathcal P$, there holds $ w (Q) \leq  C_d ^3  w (Q') $, we conclude that 
\begin{equation*}
\int _{Q}\int _{Q} \frac { \textup{dist}(x', H+x)  } { \lvert  x-x'\rvert  }  \; w (dx)\, w (dx') \ge \eta (C_d) w (Q) ^2 , 
\end{equation*}
as required.  

Second, if $ w$ is Ahlfors-David regular, namely satisfying \eqref{e:AD}, with $ n-1 < d \leq n$, we argue by contradiction 
that it is of uniformly full dimension.  Namely we assume that the inequality \eqref{e:ufd} fails for some cube $ Q$, for a sufficiently small $ 0< \eta <1$ 
specified below.  

We can select an $ x\in Q$ and a hyperplane $ H$ so that 
\begin{equation*}
\int _{Q} \frac { \textup{dist}(x', H+x)  } { \lvert  x-x'\rvert  }  \;  w (dx') \le 2 \eta w (Q). 
\end{equation*}
Let $ \tilde H = \{ x' \in Q \::\: \textup{dist}(x' , H+x) < 4 \sqrt \eta \ell (Q) \}$ be a neighborhood of $ H+x$.  
It follows that  $ w (Q \setminus \tilde H) < \sqrt \eta w (Q)$.  On the other hand, we can take a cover $ \mathcal C$  of $ \tilde H$ 
by balls of radius $ 2 \sqrt \eta \ell(Q)$. Clearly, we can assume that $ \mathcal C \le c \eta ^{- \frac {n-1}2}$.  Using both 
inequalities in the Ahlfors-David assumption \eqref{e:AD}, we then have 
\begin{align*}
C_0 \ell (Q) ^{d}  & \le w (Q) 
\\
&\le  w (Q \setminus \tilde H)  + w (\tilde H) 
\\
&\le \sqrt \eta w (Q) + \sum_{B\in \mathcal C} w (B) 
\\
&\le  \sqrt \eta w (Q) 
 + c  \eta ^{- \frac {n-1}2} \sup _{B\in \mathcal C} w (B) 
\\
& \le C_1 
\Bigl\{
\sqrt {\eta }  + c' \eta ^{- \frac {n-1}2 + \frac d 2  }
\Bigr\} \ell (Q) ^{d}  . 
\end{align*}
By the assumption $ n-1 < d \leq n$, it follows that both exponents on $ \eta $ above are positive.  
A contradiction is found for a sufficiently small $ 0< \eta <1$, as a function of $ 0< C_0 < C_1 < \infty $.  
This completes the proof.

\subsection{An Example}\label{s:example}

We discuss the proof of the energy lemma. 
The extra condition \eqref{e:ufd} is used to deduce \eqref{e:xx},  although that inequality is not directly 
proved rather, the variance inequality \eqref{e:kappa} is shown to fail for sufficiently small $ \kappa $. 

If $ \sigma $ is Lebesgue measure, restricted to some affine hyperplane of dimension less or equal to $ d$, 
then at some distance from the hyperplane, the transform $ \mathsf R _{\sigma } 1 $ is essentially constant 
in directions parallel to the hyperplane.  We show here that there are worse examples: 
For certain choices of $ d$,  one can construct $ \mathsf R _{\sigma }$ so that it is essentially zero on a unit square of co-dimension $ 1$.  

\begin{proposition}\label{p:example} 
Let $ n$ be an integer, and $ \max\{0,n-2\}< d\leq n$.  
Let $ S \subset \mathbb R ^{n}$ be a  cube of co-dimension one.   
For all $ \epsilon >0$ there is a finite weight $ \sigma $ so that 
\begin{gather} \label{ex:supp}
\textup{dist}( S, \textup{supp} (\sigma )) \geq  \epsilon ^{-1}  \ell(S) 
\\  \label{ex:R}
\epsilon ^{-1} \sup _{x\in S} \lvert  \mathsf R _{\sigma } \mathbf 1 (x) \rvert\le  
\int _{\mathbb R ^{n}} \frac {  \ell (S)} {  \ell (S) ^{d+1}+ \textup{dist}(S,x) ^{d+1}} \; \sigma(dx).
\end{gather}
\end{proposition}

The first condition says that $ \sigma $ is not supported near $ S$, and the second shows that 
the strongest norm we can place on $ S$, the $ L ^{\infty }$-norm, does not dominate the gradient Poisson average of $ \sigma $. 
 We thank Xavier Tolsa for a suggestion that lead to this example. 

\begin{proof}
Write $ \mathbb R ^{n}= \mathbb R ^{n-1} \times \mathbb R $, and write coordinates as $ (y,z)$, for $ y\in \mathbb R ^{n-1}$ and $ z\in \mathbb R $.  
A constant $ \lambda >1$ will be chosen.  Take the  cube $ S$ to be 
$ [- \lambda ^{-1/2}, \lambda ^{-1/2}] ^{n-1} \times \{0\}$.   
The weight $ \sigma $ is supported on the union of the two hyperplanes $ \Lambda  = \bigcup _{\theta \in \{-,+\}}  \Lambda _{\theta }$, 
where $ \Lambda _{\theta } = \mathbb R ^{n-1} \times \{ \theta \lambda \}$. 

We take $ \sigma $ to be translates of the same radial weight $ \tilde \sigma $ on $ \mathbb R ^{n-1}$ to the hyperplanes $ \Lambda _{\pm}$. 
Then, by oddness of the kernel for the Riesz transform,  it follows that $ \mathsf R _{\sigma } 1 (0)=0$.  
Write the coordinates of the Riesz transform by $ \mathsf R _{j, \sigma }$, for $ 1\le j\le n$. 
For $ 1\le j  \le n$, and $ 1\le j\neq k <n$,  we have 
\begin{equation*}
\frac {\partial} {\partial _{k}}  \mathsf R _{j, \sigma } 1 (0) = - (d+1)
\int \frac { y_j y_k} { \lvert  z^2 + \lvert  y\rvert ^2  \rvert ^{ \frac {d+3}2} } \; \sigma (dy \, dz) =0
\end{equation*}
by the radial property of $ \tilde \sigma $.  
For $ 1\le j < n$, we require 
\begin{equation*}
\frac {\partial} {\partial _{j}}  \mathsf R _{j, \sigma } 1 (0) = 
\int _{\Lambda } \frac { z ^2  + \lvert  y\rvert ^2  - (d+1) y_j ^2  } { \lvert  z^2 + \lvert  y\rvert ^2  \rvert ^{ \frac {d+3}2} } \; \sigma (dy\, dz) =0 .
\end{equation*}
By symmetry in $ y$, we can recognize the integral involving  $ y_j$ as $ \frac 1 {n-1}$ times the integral involving $ \lvert  y\rvert ^2  $. Thus, 
the equality above reduces to 
\begin{equation} \label{ex:1}
\int _{ \mathbb R ^{n-1}} \frac {\lambda ^2 } { \{ \lambda ^2 + \lvert  y\rvert ^2 \} ^{\frac {d+3}2} } \;\tilde \sigma (dy)
=  \frac { d+2-n} {n-1}   \cdot  \int _{ \mathbb R ^{n-1}} \frac {\lvert  y\rvert  ^2 } { \{ \lambda ^2 + \lvert  y\rvert ^2 \} ^{\frac {d+3}2} } 
\;\tilde \sigma (dy). 
 \end{equation}
 This is only possible for $ n-2<d\leq n$.  
 The conditions we place on $ \sigma $ are the symmetry properties already described, the equality \eqref{ex:1} above, and 
 \begin{align} \label{ex:2}
\int _{\mathbb R ^{n}} \frac { 1} {  1+ \textup{dist}(S,x) ^{d+1}} \; \sigma(dx) =1 
\end{align}
And, one can check that under these assumptions 
\begin{align*}
\frac {\partial} {\partial _{n}}  \mathsf R _{n, \sigma } 1 (0) 
&= \int _{\Lambda }  
 \frac { z ^2  + \lvert  y\rvert ^2  - (d+1) z^2  } { \{\lvert  z^2 + \lvert  y\rvert^2\}    ^{ \frac {d+3}2} } \; \sigma (dy\, dz) 
\\
&= (n-d-1) \int _{\Lambda }   \frac 1  { \{  z^2 + \lvert  y\rvert ^2 \}    ^{d+1} } \; \sigma (dy\, dz) .  
\end{align*}

To verify \eqref{ex:R}, namely that the Riesz transforms are uniformly small.  Observe that  
\begin{equation*}
\lvert  \nabla ^{2} \mathsf R _{\sigma } 1\rvert \lesssim 1,  
\end{equation*}
as follows by inspection, and the normalization \eqref{ex:2}.  
It follows then from Taylor's theorem, that $ \lvert  \mathsf R _{\sigma }1 (x)\rvert \lesssim \lambda ^{-1} $, for all $ x\in S$,  
since the side length of $ S$ is $ \lambda ^{-1/2}$. On the other hand, by the selection of $ \Lambda $,  
 \begin{equation*}
\int _{\mathbb R ^{n}} \frac { \lambda ^{-1/2}} {   \lambda ^{-\frac {d+1}2} + \textup{dist}(S,x) ^{d+1}} \; \sigma(dx) 
\simeq \lambda ^{-1/2}.  
\end{equation*}
and so for $ \lambda ^{-1/2} \simeq \epsilon $, we have verified \eqref{ex:R}.
\end{proof}

\section{Monotonicity} 

We discuss how to use energy to dominate off-diagonal inner products, 
which turns on two points. First, we must impose smooth truncations on the 
Riesz transforms. Second, an analysis of the off-diagonal inner products 
will reveal a more complicated monotonicity principle than appears 
in the case of the Hilbert transform ($ d=n=1$) or the Cauchy transform ($ d=1, n=2$).

This is the lemma in which we dominate the off-diagonal terms by those 
which involve positive quantities.

\begin{lemma}\label{l:monotone} Let $P$ be a cube and  $ Q$ a cube  with $10 Q\subset P$. 
Then for all functions $ g \in L ^2 (\mathbb{R}^n;w)$, supported on $ Q$, and of $ w$-mean zero,
and functions $f \in L ^2 (\mathbb{R}^n;\sigma )$ which are \emph{not} supported on $ P$, 
\begin{gather}\label{e:mono}
\begin{split}
\lvert\left\langle \mathsf R _{\sigma }  f  , g  \right\rangle _w \rvert 
&\lesssim 
\mathsf P_{\sigma} ^{\textup{g}} (   f, Q )\biggl\vert  
\biggl\langle  \frac {x} {\ell (Q)}, g \biggr\rangle_{w} \biggr\vert
 + \mathsf P ^{\textup{g}+} _{\sigma } (   f, Q ) E (w,Q) w (Q) ^{1/2} \lVert g\rVert_w .
\end{split}
 \end{gather}
\end{lemma}

Note in particular that the first term on the right is quite economical, involving only a 
part of the energy term associated to the Haar support of $ g$ in $ L ^2 (\mathbb{R}^n;w)$.  
The second term however has the full energy term, but with a Poisson term with degree 
one more than $ \mathsf P_{\sigma} ^{\textup{g}}$, as defined in \eqref{e:Pg+}. 
It will be the source of  complications in the  subsequent parts of the argument, although these turn out to not be too severe, due to the following trivial bound:  Let $ Q \subset  P \subset F$, and $ \textup{dist}( Q, \partial P) \geq (\ell(Q)) ^{\epsilon } (\ell(P)) ^{1- \epsilon }$.  Then, 
\begin{align} \label{e:g+<g} 
 \mathsf P ^{\textup{g}+} _{\sigma } ( F \setminus P , Q ) 
 \lesssim \Bigl[ \frac { \ell(Q)} {\ell(P)} \Bigr] ^{1- \epsilon }  \mathsf P ^{\textup{g}} _{\sigma } ( F \setminus P , Q ) .
\end{align}
This follows immediately from the definitions \eqref{e:Pg} and \eqref{e:Pg+}.

\begin{proof}
Since $ f$ is not supported on the cube $10 Q$, it follows that $ \mathsf R _\sigma f$ is 
a $ C ^2 $ function on $ Q$. Thus, 
for $ x\in Q$ and $ [x] ^{\sigma }_Q$, we have 
\begin{align*}
\mathsf R _\sigma f (x) - \mathsf R _\sigma f ([x] ^{\sigma }_Q) 
& =      \nabla \mathsf R _\sigma f([x] ^{\sigma }_Q) \cdot (x-[x] ^{\sigma }_Q) 
\\& \qquad 
+\tfrac{1}{2} (x-[x] ^{\sigma }_Q) ^{t} \cdot \nabla ^2 \mathsf R _\sigma f (x') \cdot (x-[x] ^{\sigma }_Q)
\end{align*}
for some point $ x'$ that lies on the line between $ x $ and $[x] ^{\sigma }_Q$.  This just depends upon a Taylor Theorem (with remainder) calculation.

By the mean zero property of $ g$, 
\begin{align*}
\langle  \mathsf R _\sigma f , g \rangle _{w} 
& = 
\langle  \mathsf R _\sigma f  - \mathsf R _\sigma f ([x] ^{\sigma }_Q), g  \rangle _{w} . 
\end{align*}
Then, by Taylor's Theorem, the right hand side is dominated by two terms.   The first of these is 
\begin{align*}
\bigl\lvert  
\langle   \nabla \mathsf R _\sigma f ([x] ^{\sigma }_Q)\cdot (x-[x] ^{\sigma }_Q) , g  \rangle _{w} 
\bigr\rvert&= \bigl\lvert 
\nabla \mathsf R _\sigma f ([x] ^{\sigma }_Q) 
\langle   x-[x] ^{\sigma }_Q , g  \rangle _{w} 
\bigr\rvert
\\& \lesssim \mathsf P ^{\textup{g}} _{\sigma } (\lvert  f\rvert , Q) 
\biggl\vert  
\biggl\langle  \frac {x} {\ell (Q)}, g \biggr\rangle_{w} \biggr\vert.  
\end{align*}
This is the first term on the right in \eqref{e:mono}. 

The second term is at most 
\begin{align*}
\sup _{x\in Q} \lvert  \nabla ^2 \mathsf R _\sigma f (x) \rvert \cdot 
 \left\lVert  \, \lvert x-[x] ^{\sigma }_Q\rvert ^2 \cdot Q\right\rVert _{w} \lVert g\rVert_w.
\end{align*}
Above, we divide $ \lvert  x-[x] ^{\sigma }_Q\rvert ^2  $ by $ \ell (Q) ^2 $, and observe that 
\begin{equation*}
\left\lVert  \frac { \lvert  x-[x] ^{\sigma }_Q\rvert ^2  } {\ell (Q) ^2  } \cdot Q\right\rVert_w 
\lesssim 
\left\lVert  \frac { \lvert  x-[x] ^{\sigma }_Q\rvert   } {\ell (Q)   } \cdot Q\right\rVert_w 
= E (w,Q) w(Q)^{\frac{1}{2}}.  
\end{equation*}
In addition, there holds 
\begin{equation*}
\ell (Q) ^2 
\sup _{x\in Q} \lvert  \nabla ^2 \mathsf R _\sigma f (x) \rvert 
\lesssim \mathsf P_{\sigma}^{\textup{g+}} (\lvert  f\rvert, Q ).  
\end{equation*}
This completes the proof. 
\end{proof}

\section{Functional Energy} 

Let $ \mathcal F \subset \mathcal D_ \sigma $ be a collection of intervals which is $ \sigma $-Carleson.  That is, for all 
$ F\in \mathcal F$, suppose that 
\begin{equation} \label{e:F-carleson}
\sum_{F' \in \mathcal F \::\: F'\subsetneq F} \sigma (F') \le \tfrac 12  \sigma (F).  
\end{equation} 
Recall the definition of $ \mathcal W_Q$ just before Proposition~\ref{p:w}.  
This is the primary tool in the proof of Lemma~\ref{l:global}, and the reader can look to the proof of that Lemma for justification for the formulation below.

\begin{theorem}[Functional Energy]
\label{t:fe}  
The Poisson operator $\mathsf P^{\textup{g}} _{\sigma } $,as defined in \eqref{e:Pg},  
satisfies this inequality. For  $ \mathcal F$ satisfying \eqref{e:F-carleson}, and 
 $ f\in L ^2 (\mathbb{R}^n;\sigma )$, 
\begin{equation}  \label{e:P<}
\sum_{F\in \mathcal F} \sum_{K\in \mathcal W_F} 
 \mathsf P^{\textup{g}} _{\sigma }  (f  \,(\mathbb R ^{n} \setminus F) ,K) ^2  
\Biggl\lVert  \sum_{\substack{ Q \in \mathcal D _{g} ^{r} \;:\; 
Q \subset K\\  \dot \pi _{\mathcal F}  Q = F}}  \Delta ^{w } _{Q} 
\frac {x} { \ell (K)} 
\Biggr\rVert_{w} ^2  \lesssim \{\mathscr E ^2 + \mathscr A_2\}   \lVert f\rVert _{\sigma } ^2 , 
\end{equation}
where $ \mathscr E$ is as in \eqref{e:E}.  
\end{theorem}

Above, note that $ \mathsf P^{\textup{g}} _{\sigma }  (f  \,(\mathbb R ^{n} \setminus F) ,K) ^2  $ is multiplied by only a part of the the term $ E (w, K) ^2 w (K)$.  In view of the form of the monotonicity estimate \eqref{e:mono}, one can see that we will need a similar estimate, but with the gradient plus Poisson.  It is 

\begin{corollary}\label{c:fe} Under the assumptions of Theorem~\ref{t:fe}, there holds 
\begin{equation}  \label{e:P+<}
\sum_{F\in \mathcal F} \sum_{K\in \mathcal W_F} 
 \mathsf P^{\textup{g}+} _{\sigma }  (f  \,(\mathbb R ^{n} \setminus F) ,K) ^2  E (w,K) ^2 w (K)
 \lesssim \{\mathscr E ^2 + \mathscr A_2\}   \lVert f\rVert _{\sigma } ^2 .
\end{equation}
\end{corollary}

\subsection{Dyadic Approximate to the Poisson} 
It is more direct to work with this form of the inequality 
\begin{gather}\label{e:Q<}
\sum_{F\in \mathcal F} \sum_{K\in \mathcal W_F} 
 \tilde {\mathsf P} _{\sigma }^{\textup{g}}  (f \, (\mathbb R ^{n} \setminus F) ,K) ^2  
\Biggl\lVert  \sum_{\substack{ Q \in \mathcal D _{g} ^{r} \;:\; 
Q \subset K\\  \dot \pi _{\mathcal F}  Q = F}}  \Delta ^{w } _{Q} 
{x} 
\Biggr\rVert_{w} ^2  \lesssim \{\mathscr E ^2 + \mathscr A_2\}  \lVert f\rVert _{\sigma } ^2 ,
\\
\textup{where} \quad 
\tilde {\mathsf P} _{\sigma }^{\textup{g}} (h, R) 
\equiv \int_{\mathbb{R}^n}  \frac {h (x)} { \ell (R) ^{d+1} + \textup{dist}(x,R) ^{d+1}} \; \sigma (dx).  
\end{gather}
Namely, the side length of $ R$ is canceled out.

Recall that  $ \mathcal D _{\sigma }$ is our dyadic grid.  
The collection of cubes $ \{ 3Q \;:\; Q\in \mathcal D _{\sigma }\}$ is the 
union of collections $ \mathcal C _{u}$, $ 1\le u \le 3 ^{n}$ 
each of which is like the dyadic grid with respect to covering and nested properties.  
This is straightforward in the case of dimension $ n=1$, and the  general case follows from this.  

It is also well-known in dimension one, that for any non-dyadic interval $ I$, there are two choice of $ 1\le u \le 3$ 
and intervals $ J_u \in \mathcal C_u$ such that $ I\subset J$ and $ \lvert  J\rvert\le 6 \lvert  I\rvert  $. 
For each  non-dyadic cube $ Q$ let $ \pi _{u} Q$ be the unique, if it exists, 
cube $L\in \mathcal C_u $ such that $ 3Q \subset L$, and $ 9 \ell (Q)\le \ell (L) \le 18 \ell (Q)$, 
so that, in the dyadic case, $ \ell (L)= 12 \ell (Q)$.   If no such cube exists, set $ \pi _{u} Q= \emptyset $. 

Then, for each $ j\in \mathbb N $, there is a  choice of $ 1\le u = u _{Q,j} \le 3 ^{n} $ 
so that  $2 ^{j}Q \subset  \pi _{\mathcal C_u} ^{j} (\pi _{u} Q) \equiv (\pi _{u} Q) ^{(j)}$.   
From this, we have 

\begin{proposition}\label{p:poisson} \cite{13120843}*{Prop 6.6} 
For $ R\in \mathcal D _{\sigma }$, and function $ \phi \ge 0$, there holds 
\begin{gather}\label{e:Papprox}
\tilde {\mathsf P} _{\sigma } ^{\textup{g}}(\phi , R) 
\simeq  
\sum_{ u \;:\; \pi _{u} R \neq \emptyset } \mathsf Q ^{\textup{g}}_{u,\sigma} (\phi , \pi _{u} R)
\\
\textup{where} \quad 
\mathsf Q ^{\textup{g}}_{u,\sigma}  (\phi , R) 
:= \sum_{j\ge 0} \frac 1 {[ 2 ^{j}\ell (R)]^{d+1}}  \int _{R ^{(j)}} \phi \; d \sigma , 
\qquad R\in \mathcal C_u. 
\end{gather}
\end{proposition}

\begin{proof}
It is clear that the Poisson term dominates the sum over the operators $ \mathsf Q ^{\textup{g}}_{u,\sigma} $. 
In the reverse direction, we have 
\begin{equation*}
\tilde {\mathsf P} _{\sigma } ^{\textup{g}}(\phi , R) 
\simeq 
\sum_{j\ge 0} 
 \ell (  2 ^{j}R) ^{-d-1} \int _{2^{j}R} \phi \; d \sigma .
\end{equation*}
But then, the proposition follows immediately from our discussion above. 
\end{proof}

Consider the two weight inequality with holes, for an operator $ \mathsf Q_{u,\sigma}^{\textup{g}}$, namely 
for non-negative constants $ \eta_R$, the inequality is 
\begin{equation}\label{e:QQ}
\sum_{R\in \mathcal C_u} 
\mathsf \mathsf Q_{u,\sigma}^{\textup{g}} (\phi \cdot R ^{c} , R) ^{2} \eta_{R} \le \mathscr Q ^2  \lVert \phi \rVert_ \sigma ^2 . 
\end{equation}
It is of course convenient to dualize this inequality. Thus, define 
$ W (R) \equiv R \times [ \ell (R)/2, \ell (R)]$, and define 
$
\eta\equiv \sum_{R\in \mathcal C_u} \eta_{R} \delta x _{W (R)} 
$, and assume that $ \eta $ is a weight, that is $ \eta (R)$ is finite for all rectangles $ R$.  
Then, \eqref{e:QQ} is equivalent to the boundedness of the bilinear form 
\begin{equation*}
\sum_{R\in \mathcal C_u}   
\mathsf Q_{u,\sigma}^{\textup{g}} (\phi \cdot R ^{c} , R) 
\int _{W (R)}  \psi \; d \eta\le \mathscr Q \lVert \phi \rVert _{\sigma } \lVert \psi \rVert _{\eta}.  
\end{equation*}
Below, $ \textup{Box}_Q \equiv Q \times [0, \ell (Q)] $.  Similar to the notation introduced in the introduction, we will let $\left\Vert \psi\right\Vert_{\eta}\equiv \left\Vert \psi\right\Vert_{L^2(\mathbb{R}^{n+1}_+;\eta)}$.

\begin{theorem}\label{t:hytP}
The inequality \eqref{e:QQ} holds if and only if these three  constants are finite 
\begin{align}\label{e:Q1}
\sup _{R\in \mathcal C _{u}}  \frac {\bigl[ \sigma (R ^{(1)} \setminus R) \cdot \eta(\textup{Box} _{R}  )\bigr]^{1/2}  } {\ell (R) ^{d+1}} 
& \equiv \mathscr Q_1  
\\ \label{e:Q2}
\sup _{R_0\in \mathcal C _{u}}  \sigma (R_0) ^{-1/2} 
 \Biggl\lVert 
\sum_{R\in \mathcal C_u \;:\; R ^{(1)}\subset R_0}     \frac {\sigma (R ^{(1)}\setminus R)}  { \ell (R) ^{d+1}} \cdot 
\textup{Box}_R \Biggr\rVert _{\eta}
  & \equiv \mathscr Q_2 , 
\\ \label{e:Q3}
\sup _{R_0\in \mathcal C _{u}} 
\eta(\textup{Box}_{R_0}) ^{-1/2} 
 \Biggl\lVert 
\sum_{R\in \mathcal C_u \;:\; R^{(1)}\subset R_0}  (  R ^{(1)} \setminus R) 
\cdot \frac { \eta ( \textup{Box}_R)}   {\ell (R) ^{d+1}} 
 \Biggr\rVert _{\sigma}
& \equiv \mathscr Q_3. 
\end{align}
Moreover, $ \mathscr Q \simeq \sum_{v=1} ^{3} \mathscr Q _{v}$.  
\end{theorem}

\begin{proof}
The three quantities are clearly necessary for the norm inequality \eqref{e:QQ}.  For the proof of sufficiency, 
we rewrite the 
\begin{equation*}
\mathsf Q ^{\textup{g}}_{u,\sigma}  (\phi , R) \simeq 
\sum_{S\supsetneq R}  \ell (S) ^{-d-1} \int _{  S ^{(1)} \setminus S} \phi \; d \sigma 
\end{equation*}
so that if we consider the associated bilinear form, for non-negative $ \psi \in L ^2 (\mathbb R ^{n+1}_+; \eta)$, 
it is 

\begin{align*}
\langle  \mathsf Q ^{\textup{g}}_{u,\sigma}  \phi , \psi \rangle _{\eta} 
&\simeq  
\sum_{R\in \mathcal C _{u}} 
\sum_{S\supsetneq R}  \ell (S) ^{-d-1} \int _{S ^{(1)} \setminus S} \phi \; d \sigma  \cdot 
\int _{W (R)} \psi \; d \eta
\\
& = 
\sum_{S\in \mathcal C_u} 
\ell (S) ^{-d-1} \int _{S ^{(1)} \setminus S} \phi \; d \sigma  
\cdot \int _{\textup{Box}_S} \psi \; d \eta. 
\end{align*}
But this is the sum of $ 2 ^{n}$ operators of the following type:  
To the dyadic grid $ \mathcal C_u$, associate the grid $ \mathcal B_u = \{\textup{Box}_Q \::\: Q\in \mathcal C_u\}$. 
To each $ Q\in \mathcal B_u$, associate two disjoint, distinguished children $ Q _{+}, Q _{-} \in \mathcal B_u$, and consider the 
bilinear form 
\begin{equation*}
\Lambda (f,g) = \sum_{Q\in \mathcal B_u} \lambda _Q \int _{Q_+} f \; d \sigma \cdot \int _{Q_-} g \; d w . 
\end{equation*}
Hyt\"onen  in \cite{13120843}*{Theorem 6.8}, shows that the boundedness of operators of this type are characterized by three conditions 
similar to \eqref{e:Q1}---\eqref{e:Q3}.  The three conditions above imply the corresponding conditions for $ \Lambda (f,g)$, hence the 
Theorem follows.  
\end{proof}

The constants $ \eta _R$ in \eqref{e:QQ} are specified as follows.  
For $ F\in \mathcal F$, and $ K\in \mathcal W_F$, set 
\begin{equation*}
\tilde \eta _K \equiv \Biggl\lVert  \sum_{\substack{ Q \in \mathcal D _{g} ^{r} \;:\; 
Q \subset K\\  \dot \pi _{\mathcal F}  Q = F}}  \Delta ^{w} _{Q} x \Biggr\rVert_{w} ^2. 
\end{equation*}
Otherwise, set $ \tilde \eta _K =0$.  Then,  for $ 1\le u \le 3 ^{n}$, and $ R\in \mathcal D _u$, set 
$ \eta _R \equiv \tilde \eta _{\pi _{u} ^{-1} R}$, provided $ \pi _{u} ^{-1} R$ is defined, and if it is not defined, set $ \eta _R =0$. 
(We suppress the dependence of $ \eta $ on $ u$.) 
Recall that $ \eta\equiv \sum_{R\in \mathcal C_u} \eta_{R} \delta x _{W (R)}$.  
With this choice of $ \eta $, the inequality \eqref{e:Q<}  implies \eqref{e:P<}.  

\smallskip 

It remains to verify the three conditions \eqref{e:Q1}---\eqref{e:Q3}, showing that each constant $ \mathscr Q_ v \lesssim \mathscr R$. 
The first of these is clearly controlled by the $ A_2$ constant, that is $ \mathscr Q_1 \lesssim \mathscr A_2 ^{1/2} $.
Indeed, for this argument, one should use the bound $ \eta _R \le  \ell (R) ^2  w (R) $, which we 
will see again below.

\subsection{Forward Testing Condition }\label{s:core}
We prove that $ \mathscr Q_2 \lesssim \mathscr E $, where the former constant is as in \eqref{e:Q2}, and we do so with a recursive argument along the stopping collection $ \mathcal F$.

The recursion is given in terms of these definitions.  
Fix an cube $ R_0 \in \mathcal C_u  $ as in \eqref{e:Q2}, 
and   let $ \mathcal R_0 (R_0)$ be those $ R\in \mathcal C_u$ such that 
$ \eta  _R \neq 0$ and $ R ^{ (1)} \subset R_0$, but   $ R$
is not contained in any $ F\in \mathcal F$ which is strictly contained in $ R_0$. 
This condition means in particular that there is a stopping interval $ F_R\in \mathcal F$, and 
a maximal good interval $ K_R \Subset \pi F$ with $ \pi _{u} K_R=R$.  
Also, set $ \mathcal R_1 (R_0)$ to be  cubes $ R\in \mathcal C_u$ such that $ \eta _R$ is 
non-zero, and $ F_R$ is a maximal stopping cube strictly contained in $ R_0$. 
We show that for $ k=0,1$,
\begin{equation} \label{e:NR0}
 N_ k (R_0) = \Biggl\lVert 
\sum_{R \in \mathcal  R_k (R_0)}   \sum_{S \;:\; R_0 \supsetneq S\supset R} 
\frac {\sigma (S ^{(1)}\setminus S)}  { \ell (S) ^{d+1}} \cdot  \textup{Box}_R
 \Biggr\rVert _{\eta}  \lesssim \mathscr E    \sigma (R_0) ^{1/2} . 
\end{equation}

The Carleson measure condition on $ \mathcal F$ will permit a recursion which completes the proof.  
Namely, let $ \mathcal F_1$ be the maximal cubes $ F\in \mathcal F$ which are contained in $ R_0$, 
and inductively set $ \mathcal F _{j+1}$ to be the $ \mathcal F$-children of the cubes in $\mathcal F_j $, then  
using a standard Cauchy-Schwarz estimate in the summing index $ j$, 
\begin{align*}
 \Biggl\lVert 
\sum_{R  \;:\; R ^{(1)}\subset R_0}     \frac {\sigma (R ^{(1)}\setminus R)}  { \ell (R) ^{d+1}} \cdot \textup{Box}_R
 \Biggr\rVert _{\eta} ^2 
 & \le 
 \sum_{k=0,1}  N_k (R_0) ^2  + \sum_{j=1} ^{\infty }  j ^2   
\sum_{F\in \mathcal F _j}  \sum_{k=0,1} N_k (F_j) ^2  
\\ & \lesssim \mathscr E ^2   \Bigl\{  \sigma (R_0) + \sum_{j=1} ^{\infty }  j ^2 \sum_{F\in \mathcal F_j} \sigma (F)
\Bigr\}
 \lesssim \mathscr E ^2   \sigma (R_0) .
\end{align*}
The geometric decay in \eqref{e:F-carleson} clearly lets us sum this series.  

To prove \eqref{e:NR0} argue first for the case of $ k=1$. 
The definition of $ \eta $, and the good property of intervals 
imply that the intervals $ \{\pi _{u}K \;:\; K\in \mathcal W_F\}$ have bounded overlaps, hence 
\begin{align*}
 N_ 1 (R_0)  
 & \lesssim 
 \sum_{R\in \mathcal R_1 (R_0)} 
 \left[
 \sum_{S \;:\; R_0 \supsetneq S\supset R} 
\frac {\sigma (S ^{(1)}\setminus S)}  { \ell (S) ^{d+1}} 
 \right] ^2 \eta _R 
 \\
 & \lesssim  
  \sum_{R\in \mathcal R_1 (R_0)} \mathsf P ^{\textup{g}} _{R} (\sigma \cdot (R_0\setminus R), R) ^2 
  E (w,R) ^2 w (R) \lesssim \mathscr E ^2 \tau (R_0) . 
\end{align*}
Here, we have passed back from the discrete approximation to the Poisson, and then used the energy inequality.

Second, argue in the case of $ k=0$, the key point is that  the collection $ \mathcal R _{0} (R_0)$ has bounded overlaps, in the sense of \eqref{e:bounded}, so that we can appeal to the energy inequality \eqref{e:E}. 
Note that  for $ R\in \mathcal R_0 (R_0)$, we have  $ R=\pi _{u}K_R$ where $ K_R\in \mathcal W_{F_R}$ for some  $ F _{R}\in \mathcal F$, with $ F_R $ not contained in $ R_0$.   
Now suppose, by way of contradiction of \eqref{e:bounded}, that  $ R_ t\in \mathcal R_0 (R_0)$, for $ 1\leq t \leq T$ satisfy 
\begin{equation*}
 \bigcup _{t=1} ^{T} (C_0-1) R_t \neq \emptyset . 
\end{equation*}
Now, we can assume that $ F _{R_t } \subset F _{R_1}$, for all $ 1\leq t \leq T$.  And, for each $ R_t$ there is a cube $ Q_t \in \mathcal W _{F_{R_1}}$ that contains it, by definition of $ \mathcal W _{F}$.  (See Proposition~\ref{p:w}.)  And, we have $ (C_0-1) R _{t} \subset C_0 Q_t$.  But the cubes $ C_0 Q_t$ have bounded overlaps, so that we see that $ T \lesssim 1$.

\subsection{Backwards Testing Condition}
We bound the term in \eqref{e:Q3}, showing $ \mathscr Q_3 \lesssim \mathscr A_2$. 
Let us first treat the case of $ n-1< d \leq n$. 
For an integer $ k\ge 0$, let $  W ^{k}_R = R \times [ 2 ^{-k-1}, 2 ^{k}] \ell (R)$,  so that $ \textup{Box}_R = \bigcup _{k=0} ^{\infty } W_R ^k$.   In \eqref{e:Q3}, we replace  $ \textup{Box}_R$ by $ W ^{k}_R$, and find geometric decay in $ k$.  

Moreover, as a matter of convienence, to each $ R\in \mathcal C_u$, we let $ \tilde R$ be one of the $ 2 ^{n}-1$ dyadic cubes of volume equal to $ R$ contained in $ R ^{(1)} \setminus R$.  
Then, using the $ A_2$ condition with holes, 
\begin{align*}
\sum_{R\in \mathcal C_u \;:\; R ^{(1)}\subset R_0}   \frac { \eta ( W_R ^k)}   {\ell (R) ^{d+1}}   &
\sum_{S\in \mathcal C_u \;:\; \tilde S \subset \tilde R } 
 \frac { \eta ( W_S ^k)}   {\ell (S) ^{d+1}}  \sigma (\tilde S) 
\\& \lesssim \mathscr A_2 2 ^{-2k} \sum_{R\in \mathcal C_u \;:\; R ^{(1)}\subset R_0}   \frac { \eta ( W_R ^k)}   {\ell (R) ^{d+1}}  
\sum_{S\in \mathcal C_u \;:\; \tilde S \subset \tilde R } \ell (S) ^{d+1} 
\\
& \lesssim \mathscr A_2 2 ^{-2k} \sum_{R\in \mathcal C_u \;:\; R ^{(1)}\subset R_0} \eta ( W_R ^k) 
\lesssim  \mathscr A_2 2 ^{-2k}  \eta (\textup{Box}_R) . 
\end{align*}
The bound on the  sum of $ \ell (S) ^{d+1}$ depends upon $ n-1 < d \leq n$. 

To argue along the same lines in the case of $ 0< d\leq n-1$, the following Lemma will complete the proof, 
but this argument depends upon the  $ A_2$ condition with no holes, and it is only at this point that 
we need this stronger $ A_2$ condition. 

\begin{lemma}\label{l:d-small} Assuming the $ A_2$ condition with no holes, the following estimate holds uniformly in $ R$:
\begin{equation}\label{e:d-small}
\sum_{S\in \mathcal C_u \;:\; S\cup \tilde S \subset \tilde R }  
 \frac { \eta ( W_S ^k)}   {\ell (S) ^{d+1}}  \sigma (\tilde S)  \lesssim 2 ^{-2k} \ell (R) ^{d+1}.  
\end{equation}
\end{lemma}

\begin{proof}
Let 
\begin{equation*}
\alpha = \frac {w (R) } {\ell (R) ^{d}} \cdot \frac {\sigma  (R) } {\ell (R) ^{d}} \leq  \mathscr A_2.  
\end{equation*}
For integers $ t \in \mathbb Z $ with $ 2 ^{t} \alpha \leq \mathscr A_2$, and integers $ u\geq 1$ let 
$ \mathcal S _{t,u}$ be the intervals $ S\subset  R$ such that $ 2 ^{u} \ell (S) = \ell (R)$, and 
\begin{equation*}
2 ^{t} \alpha \leq \frac {w (S) } {\ell (S) ^{d}} \cdot \frac {\sigma  (S) } {\ell (S) ^{d}}  < 2 ^{t+1} \alpha . 
\end{equation*}
Now, for each $ S\in \mathcal S _{t,u}$, we have either 
\begin{equation*}
\frac {w (S) } {\ell (S) ^{d}} \geq \sqrt { w (R) / \sigma  (R)}  2 ^{t/2} \sqrt \alpha 
\quad \textup{or} \quad 
\frac {\sigma (S) } {\ell (S) ^{d}} \geq \sqrt {\sigma (R)/w(R)}  2 ^{t/2} \sqrt \alpha .  
\end{equation*}
Assume that the former condition holds for at least half of the cubes in $ \mathcal S _{t,u}$.   This case suffices for the argument since the complementary case can be handled via similar methods. 

From this, we deduce an upper bound on the cardinality of $ \mathcal S _{t,u}$.  Namely, we have 
\begin{align*}
\sqrt { w (R) / \sigma(R)}  2 ^{t/2} \sqrt \alpha  [ 2 ^{-u} \ell (R)] ^{d}\cdot   {} ^{\sharp} \mathcal S _{t,u} 
& \lesssim 
\sum_{S\in \mathcal S _{t,u}} w (S) \le w (R). 
\end{align*}
From this, it follows that 
$ {}^{\sharp} \mathcal S _{t,u} \lesssim 2 ^{-t/2 + u d}  $. 

Hence, we have 
\begin{align*}
\sum_{S\in \mathcal S _{t,u} }  
 \frac { \eta ( W_S ^k)}   {\ell (S) ^{d+1}}  \sigma (\tilde S)  
 &\lesssim 2 ^{-2k +t  - u (d+1)} \alpha \ell (R) ^{d+1} \cdot  {}  ^{\sharp} \mathcal S _{t,u} 
   \\ 
   &  \lesssim 2 ^{-2k +t/2 - u} \alpha \ell (R) ^{d+1}   . 
\end{align*}
This is summable in the $ u\ge 1$ and $ t \in \mathbb Z $ such that $ 2 ^{t} \alpha \le 2 \mathscr A_2$ to the estimate we need. 

\end{proof}

\subsection{Proof of Corollary~\ref{c:fe}}

The proof of \eqref{e:P+<} is a much easier fact than \eqref{e:P<}.  Recall that the energy term $ E (w, K)$ is defined in \eqref{e:E}, and that in particular, for $ F\in \mathcal F$, 
\begin{equation*}
E (w, K) ^2 w (K) = 
\sum_{t=1} ^{\infty }  
\Biggl\lVert  \sum_{\substack{ Q \in \mathcal D _{g} ^{r} \;:\; 
Q \subset K\\  \dot \pi _{\mathcal F} ^{t}  Q = F}}  \Delta ^{w } _{Q} 
\frac {x} { \ell (K)} 
\Biggr\rVert_{w} ^2  . 
\end{equation*}
We will have a geometric decay in the variable $ t$, which measures how far down the $ \mathcal F$-stopping tree the cube $ Q$ is.

Using the definition of $ \mathcal W_F$, see Definition~\ref{d:W}, and \eqref{e:g+<g}, with $ t \geq 1$ fixed,  
below, we can replace $ \mathsf P^{\textup{g}+} _{\sigma }$ by $ \mathsf P^{\textup{g}} _{\sigma }$, gaining a factor of  $ 2 ^{-t}$.  Indeed,
\begin{align*}
\sum_{F\in \mathcal F} \sum_{K\in \mathcal W_F} &
 \mathsf P^{\textup{g}+} _{\sigma }  (f  \,(\mathbb R ^{n} \setminus F) ,K) ^2  
\Biggl\lVert  \sum_{\substack{ Q \in \mathcal D _{g} ^{r} \;:\; 
Q \subset K\\  \dot \pi _{\mathcal F} ^{t}  Q = F}}  \Delta ^{w } _{Q} 
\frac {x} { \ell (K)} 
\Biggr\rVert_{w} ^2  
\\
& \lesssim 2 ^{-t} 
\sum_{F\in \mathcal F} \sum_{K\in \mathcal W_F} 
 \mathsf P^{\textup{g}} _{\sigma }  (f  \,(\mathbb R ^{n} \setminus F) ,K) ^2  
\Biggl\lVert  \sum_{\substack{ Q \in \mathcal D _{g} ^{r} \;:\; 
Q \subset K\\  \dot \pi _{\mathcal F} ^{t}  Q = F}}  \Delta ^{w } _{Q} 
\frac {x} { \ell (K)} 
\Biggr\rVert_{w} ^2
\lesssim 2 ^{-t} \{\mathscr E ^2 + \mathscr A_2 \} \lVert f\rVert_ \sigma ^2 . 
\end{align*}
The last inequality is an instance of \eqref{e:P<}.  Clearly we can sum this over $ t\geq 1$.

\section{The Global to Local Reduction} 

\subsection{Initial Steps}
We begin the task of proving the norm boundedness of the Riesz transforms, assuming the 
$ A_2$, energy and testing hypotheses.  

\begin{lemma}\label{l:absorb}  Assume the \emph{a priori} inequality \eqref{e:N}. 
For all $ 0 < \vartheta < 1$, and choices of $ 0< \epsilon < (4d+4) ^{-1} $, 
 there is a choice of $ r$ sufficiently large so that,   
\begin{equation} \label{e:absorb}
\mathbb E \lvert  \langle  \mathsf R _{\sigma } (P ^{\sigma } _{\textup{good}}f),  P ^{w} _{\textup{good}} g \rangle _{w}\rvert 
\le \{C _{\epsilon , r, \vartheta } \mathscr R + \vartheta  {\mathscr N}\} \lVert f\rVert_{\sigma } \lVert g\rVert_{w} .
\end{equation}
It follows that $ \mathscr N \lesssim \mathscr R$,  where $ \mathscr R $ is defined at the end of Theorem~\ref{t:exact}.  
\end{lemma}

To prove \eqref{e:absorb}, we can assume that $ f$ and $ g$ are supported on a fixed cube $ Q ^{0}$.  After trivial application of 
the testing inequality, we can further assume that $ f$ and $ g$ are of mean zero in their respective spaces.  
With probability one, there is a cube $ Q ^{0} _ \sigma  \in \mathcal D_\sigma$ which 
contains $ Q ^{0}$.  Then, 
\begin{equation*}
f = \sum_{\substack{Q\in \mathcal D_\sigma\\ Q\subset Q ^{0} _{\sigma } }} \Delta ^{\sigma } _{Q} f. 
\end{equation*} 
The function $ g$ satisfies an analogous expansion.   
  
Define the bilinear form 
\begin{equation}\label{e:above}
B ^{\textup{above}} (f,g) \equiv 
\sum_{\substack{P\in \mathcal D_f  }}
\sum_{\substack{Q\in \mathcal D_{g}\\ Q\Subset _{4r} P_Q }} 
[ \Delta ^{\sigma } _{P} f ]  ^{\sigma } _{P_Q} 
\langle  \mathsf R _{\sigma }  P_Q , \Delta ^{w} _{Q} g  \rangle_w, 
\end{equation}
and define $ B ^{\textup{below}} (f,g)$ similarly.  
Here $ Q\Subset _{4r} R$ means that $ Q\cap   R \neq \emptyset $ and $ 2 ^{4r}\ell (Q) \le \ell (P)$. 
But, $ Q\in \mathcal D_{w}$ must be good, hence $ Q\subset R$, and is a relatively long way 
from the boundary of any child of $ R$.  The cube  $ P_Q$ is the child of  $ P$ that contains $ Q$. 
We have also simply written $ \mathsf R _{\sigma }$ above, suppressing the truncations.

A basic estimate is 

\begin{lemma}\label{l:2above} Under the hypotheses of Lemma~\ref{l:absorb}, there holds 
\begin{equation}\label{e:2above}
\begin{split}
\mathbb E \bigl\lvert
\langle  \mathsf R _{\sigma } (P ^{\sigma } _{\textup{good}}f),  P ^{w} _{\textup{good}} g \rangle _{w}
- &
B ^{\textup{above}} (P ^{\sigma } _{\textup{good}}f,P ^{w} _{\textup{good}} g) 
\\&
- B ^{\textup{below}} (P ^{\sigma } _{\textup{good}}f,P ^{w} _{\textup{good}} g)
\bigr\rvert  
\le \{C _{\epsilon , r, \vartheta } \mathscr R + \vartheta \mathscr N\} \lVert f\rVert_{\sigma } \lVert g\rVert_{w}.
\end{split}
\end{equation}
\end{lemma}

The proof of this lemma includes several elementary estimates, and critically, the surgery estimate, Lemma~\ref{l:surgery}, 
which requires the expectation above.  It remains to consider the form $ B ^{\textup{above}} (f,g)$, and its dual.  
This is indeed main point.

\begin{lemma}\label{l:above} For almost every choice of $ \mathcal D_\sigma$ and $ \mathcal D _{w}$, there holds 
\begin{equation}\label{e:above<}
\lvert  B ^{\textup{above}} (f,g)\rvert \lesssim \mathscr R \lVert f\rVert_{\sigma } \lVert g\rVert_{w} . 
\end{equation}
The same estimate holds for the dual form $ B ^{\textup{below}} (f,g)$.  
\end{lemma}

In the proof, only the existence of the cube $ Q ^{0}_{\sigma}  $ is required of $ \mathcal D_\sigma$.  
Hence, it suffices to assume that $ f$ and $ g$ are good, that is 
 $  P ^{\sigma } _{\textup{good}} f=f$, 
and moreover that there is an integer 
$ 0\le i_f < 4r$, for which 
\begin{equation} \label{e:Df} 
\Delta ^{\sigma } _{Q} f \not\equiv0 \quad \textup{implies} \quad 
 i_f = \log_2\ell (Q) \mod 4r .  
\end{equation}
Impose the same assumptions on $ g$, with an integer $ 0\le i_g < 4r$.  
By passing to a larger cube, we can  assume that $ \log_2 \ell  (Q ^{0} _{f}) = i_f -1 \mod r$.  
Then, let 
\begin{align}\label{e:4r}
\begin{split}
\mathcal D _{f} ^{4r} & \equiv 
\{ Q \::\:  \log_2 \ell (Q) = i_f \mod 4r\}, 
\\
\mathcal D _{f} ^{r} & \equiv 
\{ Q \::\:  \log_2 \ell (Q) = i_f -1 \mod r\}.  
\end{split}
\end{align}
In particular, the grid $ \mathcal D _{f} ^{r}  $ contains all the children of the cubes in $ \mathcal D _{f} ^{4r}$. 
Let $ \mathcal D _{g} ^{s}$, for $ s=r,4r$, have the corresponding definition.

\subsection{Stopping Data}\label{s:stopping}
Our next task is to make the \emph{global to local reduction}, which is phrased in terms of this important 
stopping time construction.  
We construct $ \mathcal F \subset \mathcal D_f ^{r}$  
in a recursive fashion.  Initialize $ \mathcal F = \{ Q ^{0} _{\sigma}\}$.
Then, in the recursive step, if $ F\in \mathcal F$ is minimal, we add to $ \mathcal F$ the maximal   
dyadic subcubes $ Q\subset F$, with $ Q\in \mathcal D  ^{r}_{f}$ such that either 
\begin{enumerate}
\item   (A big average) $ \left[\,\lvert f\rvert\,\right] _{Q} ^{\sigma } \ge  4 [\, \lvert  f\rvert\,] _{F} ^{\sigma } $,
\item  (Energy is big) $ \sum_{K\in \mathcal W_Q} \mathsf P_{\sigma}^{\textup{g}} (F \setminus C_0 K , K) ^2 E (K, w) ^2 w (K) 
\ge 10  \mathscr R ^2 \sigma (Q)
$. 
\end{enumerate}
In the second condition,  recall that $ \mathscr E \le \mathscr R$.  And, $ \mathcal W _{Q} \subset \mathcal D ^{r} _{f }$ 
are the maximal cubes $ Q'\subset Q$ such that $ \textup{dist}(Q',\partial Q) \ge \ell (Q') ^{\epsilon } \ell (Q) ^{1- \epsilon }$, which have the bounded overlaps property of Proposition~\ref{p:w}.  Namely, the energy inequality \eqref{e:E} will hold.

It is elementary to see that the collection $ \mathcal F$ is $ \sigma $-Carleson: 
\begin{equation*}
\sum_{\substack{F' \in \mathcal F\\ \textup{$ F'$ a child of $ F$}  }} \sigma (F) \le \tfrac 12 \sigma (F), \qquad F\in \mathcal F. 
\end{equation*}
 Define projections 
\begin{align}\label{e:PFs}
\begin{split} 
P ^{\sigma } _{F} f & \equiv 
\sum_{P \;:\; \pi _{\mathcal F} P= F} \Delta ^{\sigma } _{P} f , 
\\
P ^{w } _{F} g & \equiv 
\sum_{Q \;:\; \dot\pi _{\mathcal F} Q= F} \Delta ^{\sigma } _{Q} f . 
\end{split}
\end{align}
In the second line, by $  \dot\pi _{\mathcal F} Q= F$, we mean that $ F \in \mathcal F$ is the minimal element such that 
$ Q \Subset _{r} F$.  

The important quasi-orthogonality bound is 
\begin{equation}\label{e:quasi}
\sum_{F\in \mathcal F} 
\bigl\{  [\, \lvert  f\rvert\,] _{F} ^{\sigma } 
\sigma (F) ^{1/2}  + \lVert P ^{\sigma } _{F} f \rVert _{\sigma } \bigr\}
\lVert  P ^{w } _{F} g\rVert_w 
\lesssim \lVert f\rVert_\sigma \lVert g\rVert_w, 
\end{equation}
as follows from orthogonality proprieties of the projections and the construction of the stopping cubes $ \mathcal F$.  
 
The stopping cubes are used to make a decomposition of the `above' form in \eqref{e:above} into a `global' and a `local' part. 
\begin{align}\label{e:Aglobal}
B ^{\textup{above}} _{\textup{global}} (f,g) \equiv &
\sum_{\substack{P\in \mathcal D_f  }}
\sum_{\substack{Q\in \mathcal D_{g}\\ Q\Subset _{4r} P_Q ,\  \pi _{\mathcal F} Q \subset  P_Q}} 
[ \Delta ^{\sigma } _{P} f ]  ^{\sigma } _{P_Q} 
\langle  \mathsf R _{\sigma }  P_Q , \Delta ^{w} _{Q} g  \rangle_w, 
\\
B ^{\textup{above}} _{\textup{local}} (f,g) \equiv  & 
\sum_{F\in \mathcal F} B ^{\textup{above}} _{\textup{local}, F} (f,g), 
\\  \label{e:Alocal} \textup{where} \qquad 
 B ^{\textup{above}} _{\textup{local}, F} (f,g) \equiv & 
\sum_{\substack{P\in \mathcal D_f  }}
\sum_{\substack{Q\in \mathcal D_{g}\\ Q\Subset _{4r} P_Q   \subsetneqq   \pi _{\mathcal F} Q = F }} 
[ \Delta ^{\sigma } _{P} f ]  ^{\sigma } _{P_Q} 
\langle  \mathsf R _{\sigma }  P_Q , \Delta ^{w} _{Q} g  \rangle_w.  
\end{align}
The global part is when the stopping parent of $ Q$ is itself contained in $ P_Q$,  while for the local part,  
we have $ P_Q$ contained in the stopping parent of $ Q$.  

\begin{lemma}[Global to Local Reduction] 
\label{l:global}
With the notation above,  
\begin{equation}\label{e:global<}
\bigl\lvert 
  B ^{\textup{above}} _{\textup{global}} (f,g) 
\bigr\rvert \lesssim \mathscr R \lVert f\rVert_\sigma  \lVert g\rVert_w.  
\end{equation}
\end{lemma}

We prove this just below, using the functional energy inequality of Theorem~\ref{t:fe}. 
Observe that this Lemma shows that the control of the form 
$ B ^{\textup{above}} (f,g)$ is then reduced to a class of \emph{local estimates}.

\begin{lemma}[Local Estimate] 
\label{l:local}
Uniformly in $ F\in \mathcal F$, we have
\begin{equation}\label{e:local<}
\lvert   B ^{\textup{above}} _{\textup{local}, F} (f,g)\rvert 
\lesssim \mathscr R 
\bigl\{ [\,\lvert f\rvert\,]_F^{\sigma} \sigma (F) ^{1/2}  + \lVert P ^{\sigma } _{F} f \rVert _{\sigma } \bigr\}
\lVert  P ^{w } _{F} g\rVert_w. 
\end{equation}
\end{lemma}

In view of the quasi-orthogonality bound \eqref{e:quasi}, this clearly completes the control of the 
form $ B ^{\textup{above}} (f,g)$.  The delicate proof is taken up in  \S\ref{s:loc}.

\begin{proof}[Proof of Lemma~\ref{l:global}] 
The global part is defined in \eqref{e:Aglobal}, and the sum is further reorganized over the stopping cubes, 
thus 
\begin{equation}  \label{e:reorg}
\sum_{F\in \mathcal F} 
\sum_{\substack{P\in \mathcal D_f  \;:\; F\subsetneqq P }}
\sum_{\substack{Q\in \mathcal D_{g}\\ Q\Subset _{4r} P_Q ,\  \pi _{\mathcal F} Q  =F }} 
[ \Delta ^{\sigma } _{P} f ]  ^{\sigma } _{P_Q} 
\langle  \mathsf R _{\sigma }  P_Q , \Delta ^{w} _{Q} g  \rangle_w. 
\end{equation}
We first argue in the case that the conditions on $ Q$ above are replaced by the stronger condition 
\begin{equation}\label{e:stronger}
 Q\Subset _{4r} P_Q ,\  \dot \pi _{\mathcal F} Q = F. 
\end{equation}
See \eqref{e:PFs} for the definition of $\dot  \pi _{\mathcal F} Q$.

We invoke the \emph{exchange argument}, which entails the following steps.  (a) Use the stopping data to sum the martingale differences on $ f$. 
(b) Restrict the argument of the Riesz transforms to stopping cubes, invoke the testing inequality on these cubes, and 
quasi-orthogonality. (c) The remaining term has the argument of the Riesz transform and the martingale difference on $ g$  
will be in off-diagonal position, letting the monotonicity principles come into play. (d) The remaining parts of the argument 
use either the energy inequality, or its functional variants (Theorem~\ref{t:fe}), to control these terms.  

For each $ F\in \mathcal F$,  
\begin{equation} \label{e:sumF}
 \biggl\lvert \sum_{P \::\:  P\supsetneqq F} [ \Delta ^{\sigma } _{P} f] _{P_F} ^{\sigma } \biggr\rvert \lesssim 
 [\, \lvert  f\rvert\,] _{F} ^{\sigma } .
\end{equation}
This is relevant to (a) above. 
Concerning (b), for $ P\supsetneq F$, we write the argument of the Riesz transform as $ P_F = F + (P_F \setminus F)$. With the argument of the 
Riesz transform being $ F$, we have 
\begin{align*}
\biggl\lvert 
\sum_{F\in \mathcal F} \sum_{P \::\:  P\supsetneq F}
[ \Delta ^{\sigma } _{P} f] _{P_F} ^{\sigma }  
\langle   \mathsf R _{\sigma } F,    P ^{w} _{F} g  \rangle_w 
\biggr\rvert
& \lesssim 
\sum_{F\in \mathcal F}  [\,\lvert f\rvert\,]_F^{\sigma} \lvert \langle  \mathsf R _{\sigma } F,   P ^{w} _{F} g \rangle _{w}\rvert 
\\
& \lesssim \mathscr R 
\sum_{F\in \mathcal F}  [\,\lvert f\rvert\,]_F^{\sigma}  \sigma (F)^{1/2}   \lVert   P ^{w} _{F} g \rVert_{w} 
\lesssim \mathscr R \lVert f\rVert_{\sigma } \lVert g\rVert_{w} .
\end{align*}
The notation \eqref{e:PFs} is used to abbreviate the sum over $ Q$.  
We first use \eqref{e:sumF}, then the testing inequality, followed by the quasi-orthogonality bound \eqref{e:quasi}.  

It remains to bound the sum below, in which we are at point  (c) of the exchange argument,  
\begin{equation}
\sum_{F\in \mathcal F} \sum_{P \::\:  P\supsetneq F} [ \Delta ^{\sigma } _{P} f] _{P_F} ^{\sigma }  
\langle   \mathsf R _{\sigma } (P_F \setminus F),   P ^{w} _{F} g  \rangle_w . 
\end{equation}
Again, we appeal to the stopping data, estimating the argument of the Riesz transform by 
\begin{equation*}
\biggl\lvert 
\sum_{P \::\:  P\supsetneq F} [ \Delta ^{\sigma } _{P} f] _{P_F} ^{\sigma }  (P_F \setminus F) 
\biggr\rvert \lesssim 
\Phi \equiv  \sum_{F\in \mathcal F} [\,\lvert f\rvert\,]_F^{\sigma} \cdot F .
\end{equation*}
By the $ \sigma $-Carleson property of $ \mathcal F$, we have that $\left\Vert \Phi \right\Vert_{\sigma}\lesssim \left\Vert f\right\Vert_{\sigma}$. Then, monotonicity  \eqref{e:mono} applies to show that 
\begin{align}
\left\lvert 
\sum_{P \::\:  P\supsetneq F} [ \Delta ^{\sigma } _{P}f ] _{P_F} ^{\sigma }  
\langle \mathsf R _{\sigma } (P_F \setminus F), P ^{w} _{F} g  \rangle_w 
\right\rvert
& \lesssim 
\sum_{K\in \mathcal W_F}  
\mathsf P_{\sigma}^{\textup{g}} ( \Phi  \cdot F^c  , K) 
\sum_{\substack{Q \;:\; \dot\pi _{\mathcal F} Q = F, Q\subset K} } 
\Bigl\lVert \Delta ^{w} _{Q}  \frac x {\ell (K)}  \Bigr\rVert _{w} 
\lVert \Delta ^{w} _{Q} g\rVert_w 
\\  \label{e:PPgg}
&+  
\mathsf P_{\sigma}^{\textup{g}+} ( \Phi  \cdot F ^{c}, K) E (w, K) w (K) ^{1/2} 
\Biggl[ \sum_{\substack{Q \;:\; \dot\pi _{\mathcal F} Q = F, Q\subset K} } 
\lVert \Delta ^{w} _{Q} g\rVert_w ^2 \Biggr] ^{1/2} 
\\ & \equiv A_F + B_F. 
\end{align}

We are at point (d) of the exchange argument. 
Let us consider the first term on the right above.  Summing over  $ F\in \mathcal F$, 
and applying Cauchy-Schwarz, 
\begin{align*}
\sum_{F\in \mathcal F} A_F 
&\le
\Biggl[ 
\sum_{F\in \mathcal F} \sum_{K\in \mathcal W_F} 
\mathsf P_{\sigma}^{\textup{g}} ( \Phi  \cdot F^c , K)^2   
\sum_{\substack{Q \;:\; \dot\pi _{\mathcal F} Q = F, Q\subset K} } 
\Bigl\lVert \Delta ^{w} _{Q}  \frac x {\ell (K)}  \Bigr\rVert _{w} ^2 
\times 
\sum_{F\in \mathcal F} \lVert P ^{w} _{F} g\rVert_ w ^2 
\Biggr] ^{1/2} 
\\& \lesssim 
 \mathscr R \lVert f\rVert_\sigma \lVert g\rVert_w. 
\end{align*}
Here,   appeal to the functional  energy estimate of Theorem~\ref{t:fe}.  But, by inspection, control of the term $ \sum_{F\in \mathcal F} B_F $ follows in a similar way from \eqref{e:P+<}.

\bigskip 
The proof this point has controlled the sum in \eqref{e:reorg}, with the   stronger condition on $ Q$ \eqref{e:stronger} imposed.  Assuming \eqref{e:stronger} fails, we have $ Q\Subset _{4r}P_Q$,  $Q\subset F $ but $ \dot \pi _{\mathcal F} Q \neq F$. This means that $ F^{(s)}=P$ for some $ 3r \leq s \leq 4r$.  Holding $ s$ fixed,  the sum we need to control is 
\begin{align*}
\sum_{F\in \mathcal F}
\bigl\lvert  [ \Delta ^{\sigma } _{F ^{ (s)}}] ^{\sigma }_{F ^{s-1}} 
\langle \mathsf R _{\sigma } (F ^{s-1}),  \tilde P ^{w} _{F} g  \rangle_w\bigr\rvert 
& \lesssim \mathscr R
\sum_{F\in \mathcal F} 
\bigl\lvert  [ \Delta ^{\sigma } _{F ^{ (s)}}]  ^{\sigma }_{F ^{s-1}} \bigr\rvert 
\sigma (F ^{(s-1)}) ^{1/2} \lVert\tilde P ^{w} _{F} g  \rVert _{w} 
\\
& \lesssim \mathscr R \lVert f\rVert _{\sigma } \Bigl[ \sum_{F\in \mathcal F}  \lVert\tilde P ^{w} _{F} g  \rVert _{w}  ^2  \Bigr] ^{1/2} , 
\\
\textup{where } \qquad \tilde P ^{w} _{F} g &= \sum_{\substack{Q \;:\; Q\Subset _{4r}P_Q \\ Q\subset F,\  \dot \pi _{\mathcal F} Q \neq F}} \Delta ^{w} _{Q} g .  
\end{align*}
Here, we have only appealed to the cube testing condition on $ \mathsf R _{\sigma }$.  It is clear that the last sum above is dominated by $ \lVert g\rVert_w ^2 $, so our proof of the Global to Local reduction is complete.

\end{proof}

\section{The Local Estimate} \label{s:loc}

We prove Lemma~\ref{l:local}.  In so doing, we hold $ F\in \mathcal F$ 
fixed throughout the proof, and we assume that  $ g= P ^{w} _{F} g$, to reduce notation.  
We bound the term in \eqref{e:Alocal}. 

The first step is a repetition of the exchange argument.  The argument 
of the Riesz transform is $ P_Q$, where $ \pi _{\mathcal F} P=F$ and $ Q\Subset _{4r} P$, and $ P_Q \subsetneqq F$. 
The conditions $ \pi _{\mathcal F} P=F$ and $ P_Q \subsetneqq F$ are present throughout this argument,  and frequently suppressed. 
Write   $P_Q =  F + (P_Q-F)$. Define a real number $ \varepsilon _Q$ by 
\begin{equation*}
\varepsilon _Q [\,\lvert f\rvert\,]_F^{\sigma} \equiv \sum_{P \;:\; Q\Subset_{4r} P} 
[ \Delta ^{\sigma } _{P} f] _{P_Q} ^{\sigma } . 
\end{equation*}  
It follows from the construction of the stopping data that we have $ \lvert  \varepsilon _Q\rvert \lesssim 1 $. 
Thus, by reordering the sum below, an appeal to the testing condition can be made to see that 
\begin{align*}
\Biggl\lvert\sum_{P } \sum_{\substack{Q \;:\; \dot\pi_{\mathcal F}Q=F \\ Q\Subset_{4r} P}}
[ \Delta ^{\sigma } _{P} f] _{P_Q} ^{\sigma }  \langle \mathsf R _{\sigma } F , \Delta ^{w} _{Q} g \rangle _{w}
\Biggr\rvert
& = 
[\,\lvert f\rvert\,]_F^{\sigma} \Bigl\lvert
\Bigl\langle \mathsf R _{\sigma } F , \sum_{\substack{Q \;:\; \dot\pi_{\mathcal F}Q=F }} \varepsilon _Q  \Delta ^{w} _{Q} g \Bigr\rangle _{w}\Bigr\rvert
\\
& \lesssim 
[\,\lvert f\rvert\,]_F^{\sigma}  \sigma (F) ^{1/2} 
\biggl\lVert  
\sum_{\substack{Q \;:\; \dot\pi_{\mathcal F}Q=F }} \varepsilon _Q  \Delta ^{w} _{Q} g 
\biggr\rVert_w 
\\&\lesssim [\,\lvert f\rvert\,]_F^{\sigma}  \sigma (F) ^{1/2} \lVert g\rVert_w.  
\end{align*}

Thus, it remains to consider the sum when the argument of the Riesz transform is $F\setminus P_Q $.  This is the \emph{stopping form}, and it will require a subtle, recursive analysis.

\begin{lemma}[Stopping Form]
\label{l:stopping} The following estimate holds 
\begin{equation} \label{e:stopping}
\Biggl\lvert   
\sum_{ \substack{ P \;:\;  \pi _{\mathcal F} P=F \\ P_Q \subsetneqq F}}  \sum_{\substack{Q \;:\; \dot\pi_{\mathcal F}Q=F \\ Q\Subset_{4r} P}}
[ \Delta ^{\sigma } _{P} f] _{P_Q} ^{\sigma }  \langle \mathsf R _{\sigma } (F \setminus P_Q) , \Delta ^{w} _{Q} g \rangle _{w}
\Biggr\rvert 
\lesssim \mathscr R \lVert f\rVert_\sigma \lVert g\rVert_w. 
\end{equation}
\end{lemma}

The analysis will combine on the one hand, a variant of an argument related to the so-called pivotal technique 
\cite{V}, 
and the other, the subtle recursion that was identified in \cite{MR3285858}. 
In neither case is the bounded averages property of the function $ f$ irrelevant. 
(We used the bounded averages property in the exchange argument.) 
Rather, it is the fact that the energy stopping condition is incorporated into the stopping data that 
is the crucial point.

The main tool is the monotonicity principle \eqref{e:mono}, which has two terms, the `gradient' and the `gradient plus' terms. 
The `easy' term is the second, and we address it in the next subsection.

\subsection{The Gradient Plus Term Dominant in Monotonicity}

\begin{lemma}\label{l:gradient} There holds 
\begin{equation}\label{e:gradientLocal}
\Biggl\lvert  
\sum_{ \substack{ P \;:\;  \pi _{\mathcal F} P=F \\ P_Q \subsetneqq F}} \sideset { } {'} \sum_{\substack{Q \;:\; \dot\pi_{\mathcal F}Q=F \\ Q\Subset_{4r} P ,\  P_Q \subset Q ^{e}}} 
[ \Delta ^{\sigma } _{P} f] _{P_Q} ^{\sigma }  \langle \mathsf R _{\sigma } (F \setminus P_Q) , \Delta ^{w} _{Q} g \rangle _{w}
\Biggr\rvert 
\lesssim \mathscr R \lVert f\rVert_\sigma \lVert g\rVert_w. 
\end{equation}
Above, by $ \sideset { } {'} \sum_ { \cdots }$, we impose the restriction that 
\begin{equation*}
 \mathsf P_{\sigma}^{\textup{g}}(F\setminus Q ^{e}, Q) ^2  
\Bigl\lVert  \Delta ^{w} _{Q}  \frac {x} { \ell (Q)}\Bigr\rVert_{w} ^2
\leq 
\mathsf P_{\sigma}^{\textup{g}+} (F \setminus Q ^{e}, Q) ^2 E (w, Q) ^2 w (Q) . 
\end{equation*}
\end{lemma}

\begin{proof}  
For each individual summand in \eqref{e:gradientLocal}, we have by assumption,  monotonicity \eqref{e:mono} and \eqref{e:g+<g}, 
\begin{align*}
\langle \mathsf R _{\sigma } (F\setminus P_Q) , \Delta ^{w} _{Q} g \rangle _{w}
\lesssim &
\mathsf P_{\sigma}^{\textup{g}+} (F \setminus P_Q, Q)  
E (w,Q) w (Q) ^{1/2} \lVert \Delta ^{w} _{Q} g\rVert_{w} 
\\ \lesssim   & 
 \Bigl[\frac {\ell (Q)} {\ell (P)}   \Bigr] ^{1- \epsilon }\mathsf P_{\sigma}^{\textup{g}} (F \setminus P_Q, Q)  
E (w,Q) w (Q) ^{1/2} \lVert \Delta ^{w} _{Q} g\rVert_{w} . 
\end{align*}
That is, we have an extra geometric decay times an energy inequality term.  Thus, we can hold the relative length of $ Q$ fixed.

And, we can then estimate as below, where we add restrictions on the relative length of $ Q$ and $ P$. 
Below, fix a choice $ P'$ of dyadic child of $ P$, which is itself not a stopping cube.  (If it were, it makes no contribution to the local form.)
For integers $ s \geq r$, and by goodness of $ Q$, 
\begin{align}
\Biggl\lvert  &
\sum_{\substack{Q \;:\;   2^{s} \ell (Q)=\ell (P) ,\  P_Q =P'}} 
[ \Delta ^{\sigma } _{P} f ] _{P'}
\langle \mathsf R _{\sigma } (F\setminus P') , \Delta ^{w} _{Q} g \rangle _{w}
\Biggr\rvert
\\& 
\lesssim  \bigl\lvert  [ \Delta ^{\sigma } _{P} f ] _{P'}\bigr\rvert 
\sum_{\substack{Q \;:\;   2^{s} \ell (Q)=\ell (P) , P_Q = P'}} 
\mathsf P_{\sigma}^{\textup{g}+} (F \setminus P', Q)  
E (w,Q) w (Q) ^{1/2} \lVert \Delta ^{w} _{Q} g\rVert_{w} 
\\ \label{e:sdecay}
& \lesssim   
2 ^{-s/2}    \bigl\lvert  [ \Delta ^{\sigma } _{P} f ] _{P'}\bigr\rvert  
\sum_{\substack{Q \;:\;   2^{s} \ell (Q)=\ell (P) \\ P_Q = P' }} 
\mathsf P_{\sigma}^{\textup{g}} (F \setminus P', Q)  
E (w,Q) w (Q) ^{1/2} \lVert \Delta ^{w} _{Q} g\rVert_{w} .  
\end{align}

By construction, the stopping parent of $ P'$ is $ F$, so that  it fails the stopping conditions given at the beginning of \S~\ref{s:stopping}.  Therefore, after an application of Cauchy-Schwarz, the sum above is at most 
\begin{align*}
\sum_{\substack{Q \;:\;   2^{s} \ell (Q)=\ell (P) \\ P_Q = P' }} &
\mathsf P_{\sigma}^{\textup{g}} (F \setminus P', Q)  
E (w,Q) w (Q) ^{1/2} \lVert \Delta ^{w} _{Q} g\rVert_{w}
\\&
\lesssim \Biggl[
\sum_{\substack{Q \;:\;   2^{s} \ell (Q)=\ell (P) \\ P_Q = P' }} 
\mathsf P_{\sigma}^{\textup{g}} (F \setminus P', Q) ^2  
E (w,Q) ^2  w (Q) 
\times \underbrace{
\sum_{\substack{Q \;:\;   2^{s} \ell (Q)=\ell (P) \\ P_Q = P' }}  \lVert \Delta ^{w} _{Q} g\rVert_{w} ^2 } _{ \equiv \gamma (P,s) ^2 }
\Biggr] ^{1/2} 
\\
& \lesssim 
\mathscr R \sigma (P') ^{1/2} \gamma (P,s).  
\end{align*}
Combining this estimate with \eqref{e:sdecay}, we have 
\begin{align*}
 2 ^{-s/2}\mathscr R \sum_{P}  
  \bigl\lvert  [ \Delta ^{\sigma } _{P} f ] _{P'}\bigr\rvert  \sigma (P') ^{1/2} \gamma (P,s) 
  & \lesssim 
   2 ^{-s/2}\mathscr R  \Bigl[ \sum_{P} \lVert  \Delta ^{\sigma } _{P} f\rVert _{\sigma } ^2 
   \times \sum_{P} \gamma (P,s) ^2 \Bigr] ^{1/2} 
   \\
   & \lesssim 2 ^{-s/2} \mathscr R \lVert f\rVert _{\sigma } \lVert g\rVert_ w . 
\end{align*}
This is summed over $ s$ to complete the proof.  

\end{proof}

\subsection{The Gradient Term Dominant}

The form that we have yet to control is 
\begin{equation}  \label{e:zabove}
\Biggl\lvert  
\sum_{P \;:\; \pi _{\mathcal F} P} \sideset {} {''} \sum_{\substack{Q \;:\; \dot\pi_{\mathcal F}Q=F \\ Q\Subset_{4r} P  }} 
[ \Delta ^{\sigma } _{P} f] _{P_Q} ^{\sigma }  \langle \mathsf R _{\sigma } (F\setminus P_Q) , \Delta ^{w} _{Q} g \rangle _{w}
\Biggr\rvert 
\lesssim \mathscr R \lVert f\rVert_\sigma \lVert g\rVert_w, 
\end{equation}
where the notation $  \sideset {} {''} \sum_ { \cdots }$ means that 
\begin{equation}\label{e:monoG}
\lvert    \langle \mathsf R _{\sigma } (F\setminus P_Q) , \Delta ^{w} _{Q} g \rangle _{w}\rvert 
\lesssim 
\mathsf P_{\sigma} ^{\textup{g}} (  F\setminus P_Q , Q )\biggl\vert  
\biggl\langle  \frac {x} {\ell (Q)}, \Delta ^{w} _{Q} g \biggr\rangle_{w} \biggr\vert. 
\end{equation} 
But, the notation above will be suppressed throughout this section.

This case is far more subtle, requiring a delicate recursion, which in turn requires 
a more elaborate notation to explain.  
The recursion is expressed in the decomposition of the bilinear form, according to the pairs of cubes.  
For this, we need this definition. 

\begin{definition}\label{d:admiss} We call a collection $ \mathcal P \subset \mathcal D ^{4r} _{f} \times \mathcal D ^{4r} _{g}$ of pairs of cubes $ (P_1, P_2) $ 
\emph{admissible} if these conditions are met. 
\begin{enumerate}
\item   $ P_2\Subset_{4r} P_1$, with $ \dot \pi _{\mathcal F} P_2 = \pi _{\mathcal F} P_2 = F$, $ P_1$ and $ P_2$ are good, and $ (P_1) _{P_2}\supset (P_2) ^{e}$. 
\item  (Convexity in $ P_1$.) 
For each $ P_2$, if $ (P_1, P_2), (P''_1, P_2) \in \mathcal P $, and $ P_1 \subset P_1' \subset P_1''$, 
with $ P_1'$ good, then $ (P_1', P_2) \in \mathcal P$.  
\end{enumerate}
We then set $ \tilde P_1 \equiv (P_1) _{P_2} $, and also set $ \mathcal P_1 \equiv 
\{ P_1 \;:\; (P_1, P_2) \in \mathcal P, \textup{for some $ P_2$}\}$, 
and we define $ \widetilde {\mathcal P}_1$ and $ \mathcal P_2$ similarly.  
\end{definition}

We then define 
\begin{equation*}
B _{\mathcal P} (f,g) \equiv 
\sum_{ (P_1, P_2) \in \mathcal P} 
[ \Delta ^{\sigma } _{P} f] _{P_Q} ^{\sigma }  \langle \mathsf R _{\sigma } (F\setminus P_Q) , \Delta ^{w} _{Q} g \rangle _{w}. 
\end{equation*}
The form in \eqref{e:zabove} is equal to $  B _{\mathcal P} (f,g)$ for an admissible choice of $ \mathcal P$. 
Set $ \mathscr B _{\mathcal P} $ to be the best constant in the inequality 
\begin{equation*}
\lvert B _{\mathcal P} (f,g)\rvert \le \mathscr B _{\mathcal P} \lVert f\rVert _{\sigma } \lVert g\rVert _{w}. 
\end{equation*}

This preparation will be useful throughout the analysis of the local term.  

\begin{lemma}\label{l:K} For $ P \in \mathcal D_f^r$, define $ \mathcal K_P$ to be the 
maximal elements of $ \mathcal D^r_f$ such that $ 10 K\subset P$. Then, each good 
cube $ Q\in \mathcal D_ w ^{4r}$ with $ Q\Subset_{4r} \pi P$, is  satisfies $ Q\Subset _{r} K$ for some $ K\in \mathcal K_P$.
\end{lemma}

\begin{proof}
One should note that many cubes in $ \mathcal K_P$ are of side length $ 2 ^{-r} \ell (P)$, 
because the maximal side length of $ K\in \mathcal K_P$ is $ 2 ^{-r} \ell (P)$.  
But, the cube $ Q$ is much smaller than this bound: $ \ell (Q) \le 2 ^{-4r+1} \ell (P)$. 
It follows from goodness that if $ Q$ and $ K\in \mathcal K_P$ intersect, and $ 2 ^{r} \ell (Q) \le \ell (K)$, 
then  we must have $ Q\Subset_r K$.  
Thus, we can assume that $ \ell (K)\le 2 ^{-3r+1} \ell (P)$ below, and this implies that 
$ \textup{dist}(K, \partial P)\le 10 \cdot 2 ^{r} \ell (K)$, by construction of $ \mathcal K_P$.

Thus, we have  (a) $ Q\cap K\neq \emptyset $; (b) $ \textup{dist}(K, \partial P)\le 10 \cdot 2 ^{r} \ell (K)$;  
(c)  $ Q\not\Subset K$ which implies $ 2 ^{r} \ell (Q) \ge \ell (K)$, and (d) $ 2 ^{4r} \ell(Q) \le \ell (P)$.  
From goodness of $ Q$ it follows that 
\begin{equation*}
\ell (Q) ^{\epsilon } \ell (P) ^{1- \epsilon } \le 
\textup{dist}(Q,\partial P) . 
\end{equation*}
This would contradict (b) if $ \ell (Q)\ge \ell (K)$, so $ \ell (Q)\le \ell (K)$, in which case we derive 
\begin{equation*}
\ell (Q) ^{\epsilon } \ell (P) ^{1- \epsilon } \le (10 \cdot 2 ^{r}  +1) \ell (K) \le (10 \cdot 2 ^{r} +1)2 ^r \ell (Q). 
\end{equation*}
But this contradicts (d).  So the proof is complete. 

\end{proof}

Next, we define the \emph{size} of $ \mathcal P$, which must be formulated with some care. 
For a cube $ P\in \mathcal D_f^r$, set $ \mathcal K_P$ to be the maximal cubes $ K\in \mathcal D_f^r$ 
such that $ 10K\subset P$. 

\begin{proposition}\label{p:Kp} A  cube $ P_2 \in \mathcal P_2$, with $ P_2\Subset_{4r} P$ 
satisfies $ P_2\Subset _{r} K$ for some $ K\in \mathcal K _{P}$.  
\end{proposition}

\begin{proof}
Observe that the conclusion is obvious if $ K$ and $ P_2$ intersect, and $ 2 ^{r}\ell (P_2) \le \ell (K)$.  
But, also, many cubes $ K\in \mathcal K_P$ satisfy $ 2 ^{r} \ell (K) = \ell (P)$, due to the fact that 
$ K\in \mathcal D ^{r} _{f}$.  
And, $ 2 ^{4r-1}\ell (P_2) \le \ell (P)$, hence the conclusion is clear for $ 2 ^{3r-1}\ell (K)\ge \ell (P)$.  

We have $ 2 ^{3r-1} \ell (K) \le \ell (P)$, which implies $ \textup{dist}(K, \partial P) \le 20 \cdot 2 ^{r} \ell (K)$. 
Then, if $ \ell (K) \le \ell (P_2)$, it follows that $ P_2 \subset K$ and 
\begin{equation*}
20 \cdot 2 ^{r} \ell (K) \ge \textup{dist}(K, \partial P) \ge \ell (K) ^{\epsilon } \ell (P) ^{1-\epsilon }
\end{equation*}
by goodness of $ P_2$.  
But this contradicts $ 2 ^{3r-1} \ell (K) \le \ell (P)$.  

Thus, $ P_2 \subsetneq K$.  And if $ P_2 \not\Subset _{r} K$, that means $ 2 ^{r } \ell (P_2) \ge \ell (K)$, 
whence, again by goodness, 
\begin{equation*}
2 ^{-r} \ell (K) ^{\epsilon } \ell (P) ^{1- \epsilon } \le 20 \cdot 2 ^{r} \ell (K). 
\end{equation*}
That means $ \ell (P) \le [ 20 \cdot 2 ^{r (1 + \epsilon )}] ^{1/ (1-\epsilon )} \ell (K)$, which again is a contradiction. 
\end{proof}

Notice that we incorporate the previous proposition into the important definition of \emph{size}.  
\begin{align} \label{e:size}
 \textup{size} (\mathcal P) ^2  &\equiv 
\sup _{\substack{K\in  {\mathcal T}_ {\mathcal P} \\ \sigma (K) >0 }} 
 \frac { \mathsf P_{\sigma}^{\textup{g}} (F \setminus K , K) ^2 } {\sigma (K)}  
 \sum_{P_2 \in \mathcal P_2 \;:\; P_2 \Subset_{r} K} 
\Bigl\lVert \Delta ^{w} _{P_2} \frac {x} {\ell (K)}  \Bigr\rVert_w^2 , 
\\ \label{e:KP}
{\mathcal T}_ {\mathcal P} & \equiv 
\bigcup \{ \mathcal K _{\tilde P_1} \;:\; \tilde P_1\in \widetilde {\mathcal P}_1\}.  
\end{align}
Note that we only `test' the size of the collection by forming a supremum over the collection 
$ \mathcal T _{\mathcal P}$.  Our care about $ \Subset _{4r}$ and $ \Subset _{r}$ has been 
designed for this proposition.

\begin{proposition}\label{p:size} There holds 
$  \textup{size} (\mathcal P) \lesssim \mathscr R$. 
\end{proposition}

\begin{proof}
Consider $ P\in \mathcal T _{\mathcal P}$ for which $ \pi _{\mathcal F}P=F$.  Then, the cube $P$ must \emph{fail} the energy stopping condition 
of \S\ref{s:stopping}.  Therefore, we can estimate first for the Poisson operator with a hole in the argument, 
\begin{align}
\mathsf P_{\sigma}^{\textup{g}} (F\setminus P, P) ^2 
\sum_{\substack{P_2 \in \mathcal P_2 \\ P_2 \Subset_{r} P} }
\Bigl\lVert \Delta ^{w} _{P_2} \frac {x} {\ell (P)}  \Bigr\rVert_w^2  
& \lesssim   
\Biggl[ \int _{F \setminus P} 
\frac 1 { \ell (P) ^{d+1}  + \textup{dist}(x,P) ^{d+1} } \;  \sigma (dx)
\Biggr] ^2  \sum_{K\in \mathcal W_P}  \sum_{\substack{P_2 \in \mathcal P_2\\ P_2 \subset K }} \lVert \Delta ^{w}_{P_2} x\rVert_{w} ^2 
\\
& \lesssim 
  \sum_{K\in \mathcal W_P}  
  \Biggl[ \int _{F\setminus P} 
\frac {\ell (K)} { \ell (K) ^{d+1}  + \textup{dist}(x,K) ^{d+1} }\; \sigma (dx) 
\Biggr] ^2 
\sum_{\substack{P_2 \in \mathcal P_2\\ P_2 \subset K }} \Bigl\lVert \Delta ^{w}_{P_2} 
\frac x {\ell (K)}\Bigr\rVert_{w} ^2 
\\
& \lesssim \textup{size} (\mathcal P) ^2 \sigma (P).  
\end{align}

If $ \pi _{\mathcal F}P\subsetneq F$, then, by admissibility, there is no cube $ P_2 \in \mathcal P_2$ 
with   $P_2 \Subset _{r} K$.  This is because otherwise we would have $ \dot\pi _{\mathcal F}P_2 \neq F$.  
Then, the inequalities above are trivial.

\end{proof}

Our task is then to show that
\begin{lemma}\label{l:2size} 
 For all admissible $ \mathcal P$,  
\begin{equation}  \label{e:2size}
\lvert B _{\mathcal P} (f,g)\rvert 
\lesssim \textup{size} (\mathcal P)\lVert f\rVert_\sigma \lVert g\rVert_w . 
\end{equation}
\end{lemma}

The main step in the proof is phrased in terms of this variant of orthogonality. 
We say that  an enumeration of admissible collections  $ \{\mathcal P^j \;:\; j\in \mathbb N \}$  is 
\emph{orthogonal} if and only if  (a) the collections of cubes $ \mathcal P _{2} ^{j}$ are pairwise disjoint, 
and (b) the collections of cubes $ \widetilde {\mathcal P} _1 $ are pairwise disjoint.  One 
should note the asymmetry in the definition, which comes from a corresponding asymmetry in the 
roles of $ P_1$ and $ P_2$ in the definition of $ B _{\mathcal P} (f,g)$. 

\begin{lemma}\label{l:ortho} Let $ \{\mathcal P^j \;:\; j\in \mathbb N \}$ be admissible and orthogonal.  
Then, there holds 
\begin{equation}\label{e:ortho}
\lvert B _{\bigcup _{j}\mathcal P_j} (f,g)\rvert \le 
\sqrt 2 \sup _{j} \mathscr B _{\mathcal P_j} \cdot \lVert f\rVert _{\sigma } \lVert g\rVert _{w}. 
\end{equation}
\end{lemma}

\begin{proof}
Notice that a give cube $ P_1$ can be in two different collections $ \widetilde {\mathcal P} ^{j}_1$, 
which fact explains the $ \sqrt 2$ above.  
Let $ \Pi ^{j} _{\sigma } f = \sum_{P_1 \in \mathcal P ^{j}_1} \Delta ^{\sigma } _{P_1} f $, 
and define 
 $ \Pi ^{j} _{w } g= \sum_{P_2 \in \mathcal P ^{j}_2} \Delta ^{w } _{P_g} g $.  
 Then, we have 
 \begin{equation*}
\sum_{j\in \mathbb N } \lVert \Pi ^{j} _{\sigma } f \rVert _{\sigma } ^2 \le 2 \lVert f\rVert _{\sigma } ^2 ,
\end{equation*}
while the same inequality holds with constant one in $ L ^2 (\mathbb{R}^n;w)$.  
Then, there holds 
\begin{align*}
\left\lvert B _{\bigcup _{j}\mathcal P_j} (f,g)\right\rvert &\le 
\sum_{j\in \mathbb N } \lvert B _{\mathcal P_j} (f,g)\rvert 
\\
& \le \mathscr B _{\mathcal P_j} \sum_{j\in \mathbb N } \lVert \Pi ^{j} _{\sigma } f\rVert _{\sigma }    
\lVert \Pi ^{w} _{j}g\rVert_w
\\
& \le 
\sup _{j} \mathscr B _{\mathcal P_j}
\Biggl[ 
\sum_{j\in \mathbb N } \lVert \Pi ^{j} _{\sigma } f\rVert _{\sigma }  ^2 
\times 
\sum_{j\in \mathbb N } \lVert \Pi ^{j} _{\sigma } 
\lVert \Pi ^{w} _{j}g\rVert_w ^2 
\Biggr] ^{1/2} 
\\
& \le \sqrt 2 
\sup _{j} \mathscr B _{\mathcal P_j} \cdot \lVert f\rVert _{\sigma } \lVert g\rVert_w . 
\end{align*}
\end{proof}

With that preparation, our main lemma provides us with a decomposition of an arbitrary admissible 
collection into `big' and `small' collections.  The big collections have a control on their operator 
norm, and we can recurse on the small collections.  

\begin{lemma}[Size Lemma]
\label{l:size}  
For any admissible $ \mathcal P$, there is 
a decomposition $ \mathcal P = \mathcal P _{\textup{big}} \cup \mathcal P _{\textup{small}}$ such that 
\begin{equation}\label{e:big}
\lvert  B _{\mathcal P _{\textup{big}}} (f,g)\rvert  
\lesssim \textup{size} (\mathcal P) \lVert f\rVert _{\sigma } \lVert g\rVert_w , 
\end{equation}
and, $ \mathcal P _{\textup{small}}$ is the union of admissible  collections 
$ \{\mathcal P _{\textup{small}} ^{j} \;:\; j\in \mathbb N \}$, with 
\begin{equation}\label{e:small}
\sup _{j\in \mathbb N } \textup{size} (\mathcal P ^{j} _{\textup{small}}) \le 
\tfrac 14 \textup{size} (\mathcal P).  
\end{equation}
Moreover, the collections $ \{\mathcal P _{\textup{small}} ^{j} \;:\; j\in \mathbb N  \}$ 
are orthogonal.  
\end{lemma}

\begin{proof}[Proof of \eqref{e:2size}]  
By recursive application of the Size Lemma, we can write $ \mathcal P = \bigcup _{t=1} ^{\infty } \mathcal P _{t}$ 
where $ \mathcal P _{t}$ is the union of orthogonal collections $ \{ \mathcal P _{t} ^{j} \;:\; j\in \mathbb N \}$, 
which satisfy $ \mathscr B _{\mathcal P _t ^{j}} \lesssim 4 ^{-t} \textup{size} (\mathcal P)$.  Thus, 
\begin{equation*}
\mathscr B _{\mathcal P} \le \sum_{t=1} ^{\infty }  \mathscr B _{\mathcal P_t} 
\lesssim \textup{size} (\mathcal P) \sum_{t=1} ^{\infty } 2 \cdot 4 ^{-t} \lesssim \textup{size} (\mathcal P). 
\end{equation*}
\end{proof}

\subsection{Decomposition of \texorpdfstring{$ \mathcal P$}{the collection}}

 Define the measure on $ \mathbb R ^{n+1}_+$ by 
\begin{equation}  \label{e:lam-def}
\lambda = \lambda _{\mathcal P} \equiv \sum_{P_2 \in \mathcal P_2} \lVert \Delta ^{w} _{P_2} {x} \rVert_w  ^2 \delta_{\tilde x_{P_2}}
\end{equation}
where $ \tilde x _{P} = (x_P, \ell (P))$.   The main condition we have is this reformulation of the definition of size:
\begin{gather}\label{e:Nsize}
\sup _{Q\in {\mathcal T}_ {\mathcal P}}  
\mathsf P_{\sigma}^{\textup{g}} (F \setminus Q ,Q) ^2 \frac {\lambda (\textup{Tent} _{Q}) } {  \sigma (Q) \ell (Q) ^{2}} 
= \mathbf S ^2 \equiv \textup{size} (\mathcal P) ^2  ,
\\
\textup{Tent} _{Q} \equiv \bigcup _{ K\in \mathcal W _{Q}}  \textup{Box}_K .  
\end{gather}
Here, we are using the notation $ \textup{Box}_K\equiv K \times [0, \ell (K))$, as it is 
used in the functional energy inequality. 

This collection is used to make the decomposition of $ \mathcal P$.  Set $ \mathcal L_0$ to be the minimal elements $Q\in {\mathcal T}_ {\mathcal P}$ such that 
\begin{equation} \label{e:L0}
\mathsf P_{\sigma}^{\textup{g}} (F\setminus Q ,Q) ^2 \frac {\lambda (\textup{Tent} _{Q}) } {   \ell (Q) ^{2}} \ge c \mathbf S ^2  \sigma (Q).  
\end{equation}
Here $ c = \frac 1 {32}$.   Such $ Q$ exist, by definition of size.  Then, for $ n\ge 1$, inductively define 
$ \mathcal L _{n}$ to be the minimal elements $ L \in {\mathcal T}_ {\mathcal P}$ such that 
\begin{equation}\label{e:recurse}
\lambda (\textup{Tent} _{L}) \ge (1+ c) 
 \sideset {} {^{\ast }}\sum_{L' \::\: L'\subsetneq  L} \lambda (\textup{Tent} _{L'}).  
\end{equation}
where the last sum is performed over the maximal elements $L' \in \bigcup _{m=0} ^{n-1} \mathcal L _{m} $, with $ L'\subsetneq L$ (this is designated by the $\ast$ appearing on the sum). 
Then set $ \mathcal L \equiv \bigcup _{n \ge 0} \mathcal L_n$.  

The collection $ \mathcal P _{\textup{small}}$ is then defined this way.  Set $ \mathcal P _{\textup{small}} ^{0}$ to be the 
collection of pairs $(P_1, P_2) \in \mathcal P  $ such that $\tilde P_1$ does not have a parent in $ \mathcal L$. 
And, for each $ L\in \mathcal L$, define 
\begin{equation} \label{e:small-L}
\mathcal P _{\textup{small}} ^{L}  \equiv \{ (P_1, P_2) \::\:  \tilde \pi _{\mathcal L} P_2 = \pi _{\mathcal L} \tilde P_1 = L,\ \tilde P_1\subsetneq L\}. 
\end{equation}
Here and below, $ \tilde \pi _{\mathcal L} P_2  $ is the minimal element of $ \mathcal L$ such that $ P_2 \subset L$ 
and $ P_2\Subset_{r} \pi L$.

\begin{lemma}\label{l:small} The collections $\mathcal P _{\textup{small}} ^{0}  $ and $ \mathcal P _{\textup{small}} ^{L}   $, 
for $ L\in \mathcal L$ are admissible, have size at most $ \tfrac 14 \textup{size} (\mathcal P) \le \tfrac 14 \mathbf S$.  Moreover, the collections $ \{ \mathcal P _{\textup{small}} ^{L} \;:\; L\in \mathcal L\}$ are orthogonal. 
\end{lemma}

\begin{proof}
Admissibility is inherited from $ \mathcal P$ and the construction of the collections.  
Orthogonality is also clear from the construction in terms of $ \mathcal L $. 
Thus, it remains to check that the collections have small size.  For $ \mathcal P _{\textup{small}} ^{0}$, 
suppose there is  a  cube $ Q\in {\mathcal T}_ {\mathcal P_{\textup{small}} ^{0}}$ such that 
\begin{align} 
 \tfrac 1 {16} \mathbf S ^2 
& \le 
\frac { \mathsf P_{\sigma}^{\textup{g}} (F\setminus Q ,Q) ^2 } {  \sigma (Q) \ell (Q) ^{2}} 
\times  \lambda _{\mathcal P_{\textup{small}} ^{0}} (\textup{Tent}_Q).  
\end{align}
If $ Q$ does not contain any element of $ \mathcal L$, we would contradict the construction of 
that collection.  Hence, it does contain elements of $ \mathcal L$, and hence summing over 
the maximal such $ L\in \mathcal L$ below, there holds 
\begin{align*}
 \tfrac 1 {16} \mathbf S ^2 
& \le c 
\frac { \mathsf P_{\sigma}^{\textup{g}} (F\setminus Q ,Q) ^2 } {  \sigma (Q) \ell (Q) ^{2}} 
\times 
\sideset {} {^{\ast} }
\sum_{ L\in \mathcal L \::\: L\subset Q}
\lambda_{\mathcal P} ( \textup{Tent})
 \le c \mathbf S ^2 .  
\end{align*}
Notice that the constant $ c$ enters in because of construction, see \eqref{e:recurse}.  We see a contradiction 
since $ c= \frac 1 {32}$. 
Thus, $ \mathcal P _{\textup{small}} ^{0}$ has small size.  

Turn to the collections $ \mathcal P _{\textup{small}} ^{L} $ as defined in \eqref{e:small-L}. 
Again, if the size is more than $ \tfrac 14 \mathbf S$, then there is a cube $ Q\in  {\mathcal T}_ {\mathcal P _{\textup{small}} ^{L}}$ such that 
\begin{align*}
\tfrac 1 {16} \mathbf S ^2 
& \le 
\frac { \mathsf P_{\sigma}^{\textup{g}} (F \setminus Q,Q) ^2 } {  \sigma (Q) \ell (Q) ^{2}} 
\lambda _{\mathcal P _{\textup{small}} ^{L}} (\textup{Tent}_Q). 
\end{align*}
Moreover, $ Q$ cannot be contained in any $ L' \in \mathcal L$ which is a descendant of $ L$ in the $ \mathcal L$-tree, 
otherwise the term on the right above is zero. It follows that the cube $ Q$ must \emph{fail} the inequality \eqref{e:recurse}. 
Hence, we can repeat the previous argument to deduce that the size of this collection is also small. 
\end{proof}

\subsection{Controlling the Big Collection}
We have completed the proof of \eqref{e:small}.  
By definition, the collection $ \mathcal P _{\textup{big}}$ is the  (non-admissible) complementary collection.  
We decompose it into the union of  two collections $ \mathcal P ^{j} _{\textup{big}}$, for $ j=1,2$, 
with the appropriate bound on the norm  of the bilinear form $ \mathscr B _{\mathcal P _{\textup{big}} ^{j}}$ in each case. 
The essential point is that in each of the big collections, the intricate relationship between $ P_1$ and $ P_2$ is moderated by 
a `separating' collection of cubes, uniformly over pairs in the big collection.  This permits the estimation of the operator norm.  

Set $ \mathcal P ^{1} _{\textup{big}} \equiv \bigcup _{L\in \mathcal L}  \mathcal P ^{1, L} _{\textup{big}}$, where the latter collection is 
\begin{equation*}
 \mathcal P ^{1, L} _{\textup{big}} \equiv \{  (P_1, P_2 ) \in \mathcal P \::\:  \tilde P_1 = L,\ \tilde \pi _{\mathcal L} P_2 = L\}.  
\end{equation*}
Observe that these collections are admissible and orthogonal.  
Moreover,   the structure of these collections is quite rigid, since $ \tilde P_1$ is a fixed interval. 

\begin{lemma}\label{l:1big} There holds, uniformly over $ L\in \mathcal L$, 
that  $ \mathscr B _{\mathcal P ^{1, L} _{\textup{big}} } \lesssim \mathbf S$.  
\end{lemma}

\begin{proof}
Set $ \mathcal Q \equiv \mathcal P ^{1, L} _{\textup{big}} $. 
We can estimate for each $ K\in \mathcal W_L$, 
\begin{align*}
\Biggl\lvert 
\sum_{ \substack{(P_1, P_2) \in \mathcal Q\\ P_2 \Subset _r K }} &
[ \Delta ^{\sigma } _{P_1}f] _{\tilde P_1} ^{\sigma }  \langle  \mathsf R _{\sigma } (F\setminus \tilde P_1), \Delta ^{w} _{P_2} g \rangle_w 
\Biggr\rvert
\\
& \lesssim \bigl\lvert [ \Delta ^{\sigma } _{P_1}f] _{\tilde P_1} ^{\sigma } \bigr\rvert  \mathsf P_{\sigma}^{\textup{g}} (F\setminus K, K)  \sum_{ \substack{P_2 \in \mathcal Q_2\\  P_2 \Subset_r K} }
\Bigl\lVert  \Delta ^{w} _{P_2}\frac x { \ell (K)} \Bigr\rVert_{w} \lVert \Delta ^{w} _{P_2}g \rVert_{w} 
\\
& \lesssim \mathbf S \, \bigl\lvert [ \Delta ^{\sigma } _{P_1}f] _{\tilde P_1} ^{\sigma } \bigr\rvert \sigma (K) ^{1/2} 
\Biggl[\,
\sum_{ \substack{P_2 \in \mathcal Q_2\\  P_2 \Subset_r K} }
\lVert \Delta ^{w} _{P_2}g \rVert_{w} ^2 
\Biggr] ^{1/2} 
\lesssim \mathbf S \lVert f\rVert _{\sigma } \lVert g\rVert_w . 
\end{align*}
The monotonicity principle applies, as formulated in \eqref{e:monoG}. 
which has both the gradient and gradient-plus Poisson terms.  
 Then one appeals to Cauchy--Schwarz, and importantly, the definition of $ \textup{size} (\mathcal P) $ to gain the term 
$ \mathbf S \sigma (K) ^{1/2} $ above.   The final inequality above is trivial.  
\end{proof}

The second collection is built of pairs that are `separated' by $ \mathcal L$.  
Now, we set $ \tilde\pi^{1} _{\mathcal L} P_2 = \tilde \pi _{\mathcal L} P_2$, and inductively define $ \tilde \pi^{t+1} _{\mathcal L} P_2$ to be the 
minimal element of $\mathcal L$ that \emph{strictly} contains $ \tilde \pi^{t} _{\mathcal L} P_2$. Then, 
$  \mathcal P ^{2} _{\textup{big}} \equiv \bigcup _{t=2} ^{\infty }   \mathcal P ^{2,t} _{\textup{big}}  $, where 
\begin{equation}\label{e:tL}
\mathcal P ^{2,t} _{\textup{big}} \equiv \bigcup _{L\in \mathcal L} \mathcal P ^{2,t,L} _{\textup{big}} 
\equiv 
\bigcup _{L\in \mathcal L} 
\{  (P_1, P_2) \in \mathcal P \::\:    \pi _{\mathcal L} \tilde P_1 = \tilde\pi _{\mathcal L} ^{t} P_2 = L \}.  
\end{equation}
These collections are admissible, a property inherited from $ \mathcal P$.  
For fixed $ t\ge 2$,  the collections $ \{\mathcal P ^{2,t,L} _{\textup{big}} \::\: L\in \mathcal L\}$ are orthogonal, 
as is very easy to see from the definition.  
We prove 

\begin{lemma}\label{l:big-t} For all $ t\ge 2$ and $ L\in \mathcal L$, there holds 
\begin{equation}\label{e:big-t}
\mathscr B _{\mathcal P ^{2,t,L} _{\textup{big}} } \lesssim (1+c) ^{-t/6} \mathbf S . 
\end{equation}
\end{lemma}

In view of the orthogonality, and Lemma~\ref{l:ortho}, 
this estimate is clearly summable in $ t\ge 2$  and $ L\in \mathcal L$ to the estimate we need.   
Namely, we will have 
\begin{align*}
\sum_{t=2} ^{\infty } 
\biggl\lvert  \sum_{L\in \mathcal L} B _{\mathcal P ^{2,t,L} _{\textup{big}}}  (f,g) \biggr\rvert
& \lesssim \lVert f\rVert _{\sigma } \lVert g\rVert_w
\sum_{t=2} ^{\infty } 
\sup _{L\in \mathcal L} \mathscr B _{\mathcal P ^{2,t,L} _{\textup{big}}}  
\\
& \lesssim \mathbf S \lVert f\rVert _{\sigma } \lVert g\rVert_w 
\sum_{t=2} ^{\infty } (1+c) ^{-t/6} \lesssim \mathbf S \lVert f\rVert _{\sigma } \lVert g\rVert_w .  
\end{align*}

\begin{proof}[Proof of Lemma~\ref{l:big-t}] 
There are two parts of the proof. First, identifying a quantity that controls the norm of the bilinear form, and second that
this quantity decreases geometrically in $ t$.   Define $ \mathcal C _{L}$ to be the $ \mathcal L$-children of $ L$, and define 
\begin{equation} \label{e:bfT}
\mathbf T ^2 \equiv 
\sup _{L' \in \mathcal C_L} 
\sup _{K\in \mathcal K _{L'}}
\mathsf P_{\sigma}^{\textup{g}} (F\setminus K,K) ^2 \frac {\lambda _{\mathcal Q} (\textup{Tent} _{K}) } {  \sigma (K) \ell (K) ^{2}} , 
\end{equation}
The first part is to show that $ \mathscr B _{\mathcal Q} \lesssim \mathbf T$, where 
to ease notation, set $ \mathcal Q = \mathcal P ^{2,t,L} _{\textup{big}} $.  

Let $ \mathcal G$ be standard stopping data for $ f$.  Since we can assume that the Haar support of $ f$ is contained in 
$ \mathcal P_2$, we can write $ \mathcal G = \bigcup _{t\ge 0} \mathcal G_t$, where $ \mathcal G_0 = \{F\}$, 
and in the inductive stage, for $ G\in \mathcal G_t$, add to $ \mathcal G_t$ the maximal cubes $ Q \subset G$, 
$ Q$ a child of a $ P_1 \in \mathcal P_1$,  
such that $ [\,\lvert f\rvert\,] ^{\sigma } _{Q} \ge 10 [\,\lvert f\rvert\,] ^{\sigma } _{G}  $.  

With this version of the stopping data, a variant of the quasi-orthogonality bound holds, which we will use. In addition, the 
collapse of telescoping sums is given by this formula.  For each $ P_2 \in \mathcal P_2$, and $ G\in \mathcal G$, there holds 
\begin{equation}\label{e:tele}
\Bigl\lvert 
\sum_{\substack{ P_1 \::\: (P_1, P_2) \in \mathcal P  }}  [ \Delta ^{\sigma } _{P_1}] ^{\sigma  } _{ \tilde P_1} (F \setminus \tilde P_1)
\Bigr\rvert 
\lesssim  [\,\lvert f\rvert\,] ^{\sigma } _{ \pi _{\mathcal G}P_2 }  \cdot F .  
\end{equation}
This depends upon admissibility, namely the  assumption of convexity in $ P_1$, holding $ P_2$ fixed.  

For  a cube $ G$ set   $ \Pi ^{w} _{G} \equiv \sum_{\substack{P_2 \in \mathcal P_2\\  P_2 \subset G}} \Delta ^{w } _{P_2}  $.  Observe that we can estimate as follows for $ G\in \mathcal G$, and $ L'\in \mathcal C_L$ such that $ \pi _{\mathcal G}L'=G$,  
\begin{align}
\Biggl\lvert 
\sum_{\substack{ (P_1, P_2 ) \in \mathcal Q\\  P_2 \Subset _{r} L' }}  &
[ \Delta ^{\sigma } _{P_1} f] ^{\sigma } _{\tilde P_1} 
\langle \mathsf R _{\sigma }  (F\setminus \tilde P_1),  \Delta ^{w} _{P_2} g \rangle_w 
\Biggr\rvert
\\
\lesssim & \sum_{K\in \mathcal K _{L'}} 
\sum_{ \substack{ P_2 \in \mathcal P_2 \\  P_2 \subset K}}  
\Biggl\lvert 
\sum_{\substack{ P_1  \in \mathcal Q\\  P_1 \supsetneq G }}  
[ \Delta ^{\sigma } _{P_1} f] ^{\sigma } _{\tilde P_1} 
\langle \mathsf R _{\sigma }  (F\setminus \tilde P_1),  \Delta ^{w} _{P_2} g \rangle_w 
\Biggr\rvert
\\  
& \lesssim 
 [\,\lvert f\rvert\,] ^{\sigma } _{G} 
 \sum_{K\in \mathcal K _{L'}} 
\sum_{ \substack{ P_2 \in \mathcal P_2 \\  P_2 \subset K}}  
\mathsf P_{\sigma}^{\textup{g}}(F \setminus\tilde P_1 , P_2) 
\biggl\lvert \biggl\langle \frac {x}{\ell (Q_2)} , \Delta ^{w} _{Q_2} g  \biggr\rangle _w\biggr\rvert
\\  \label{e:aboveQe} 
& \lesssim \mathbf T 
 [\,\lvert f\rvert\,] ^{\sigma } _{G} 
 \sum_{K\in \mathcal K _{L'}} 
\sigma (K) ^{1/2}  \lVert \Pi ^{w} _{K} g\rVert_{w} \lesssim \mathbf T 
 [\,\lvert f\rvert\,] ^{\sigma } _{G} \sigma (L') ^{1/2} \lVert \Pi ^{w} _{L'} g\rVert_{w} .
\end{align}
We are combining several observations: (a) for each $ P_2 \in \mathcal Q_2$, we have $ P_2 \Subset _{r} L'$ for some $ L'\in \mathcal C_L$, 
and $ \pi _{\mathcal G} P_2 = \pi _{\mathcal G} L'$;  
(b) and then, $ P_2 \subset K$ for some $ K\in \mathcal K _{L'}$ thus the first inequality above holds; (c) 
we have convexity in $ P_1$, hence we can use \eqref{e:tele}, giving us the stopping value above;  (d) we can also apply the 
monotonicity principle  \eqref{e:monoG}; (e) 
and a trivial use of Cauchy--Schwarz in $ P_2$, with the definition of $ \mathbf T$ in \eqref{e:bfT} gives us the concluding inequality above.  

Notice that an application of Cauchy--Schwarz, and quasi-orthogonality shows that 
\begin{equation*}
\sum_{G\in \mathcal G} \sum_{L' \in \mathcal C \::\: \pi _{\mathcal G} L'=G}   \eqref{e:aboveQe} 
\lesssim \mathbf T \lVert f\rVert_{\sigma } \lVert g\rVert_{w} .  
\end{equation*}

\bigskip 
We turn to the second half of the proof of \eqref{e:big-t}, namely that $ \mathbf T \lesssim (1+c) ^{-t/2} \mathbf S$, 
where $ \mathbf T$ is as in \eqref{e:bfT}.   It suffices to assume that $ t \ge 10$ say. 
Take a cube $ K\in \mathcal K _{L'}$, for $ L'\in \mathcal C _{L}$.  
It follows from Proposition~\ref{p:close} below, that $ \pi _{\mathcal L} ^{u} K= L$, for $ u\le 5$. 
Notice that by construction of $ \mathcal L$, and simple induction, that for  
 integer $ s = r - 5$,  
\begin{align*}
\lambda _{\mathcal Q} (K) 
& \le 
\sum_{\substack{L''\in \mathcal L \::\:  L''\subset K \\ \dot\pi^{s} _{\mathcal L}L''=L}} \lambda (\textup{Tent} _{L''}) 
\\
&\le (1+ c ) ^{-1} 
\sum_{\substack{L''\in \mathcal L \::\:  L''\subset K \\ \dot\pi^{s-1} _{\mathcal L}L''=L}}\lambda (\textup{Tent} _{L''}) 
\\&\;\;\vdots\\& 
\le (1+ c  ) ^{-s+1} 
\sum_{\substack{L''\in \mathcal L \::\:  L''\subset K \\ \dot\pi^{1} _{\mathcal L}L''=L}}
\lambda (\textup{Tent} _{L''}) 
\lesssim  (1+ c  ) ^{-t/2} \lambda (\textup{Tent}_K).  
\end{align*}
Note that we have used \eqref{e:recurse} a total of $ s-1$ times to get the conclusion above.  Therefore, 
by Proposition~\ref{p:close}, 
\begin{align*}
(1+ c ) ^{s}  \frac {\mathsf P_{\sigma}^{\textup{g}}(F \setminus K,K)^2} {\ell (K)^2}  \lambda _{\mathcal Q} (K)
\lesssim  
 \frac {\mathsf P_{\sigma}^{\textup{g}}(F \setminus K,K)^2} {\ell (K)^2}  
 {\lambda (\textup{Tent}_K)}   \lesssim \mathbf S ^2 \sigma(K) \,. 
\end{align*}
This concludes our proof. 
\end{proof}

The Lemma above  depends upon properties of the collection $ {\mathcal T}_ {\mathcal P} $.  
\begin{proposition}\label{p:close}
Let $ Q\in \widetilde {\mathcal P}_1$, with $ \pi _{\mathcal L}Q=L$, 
and that $ K\in K_Q$  
is a cube on which we are testing the size of $\mathcal P ^{2,t,L} _{\textup{big}} $. 
Then $ Q\subset \pi^{4} _{\mathcal L} K$.  
\end{proposition}

\begin{proof}
Now, $ \ell (\pi^{3} _{\mathcal L} K) \ge 2 ^{3r} \ell (K)$, by our construction, so 
the conclusion is obvious if $ \ell (K)\ge 2 ^{-3r} \ell (Q)$.  

Continuing under the hypothesis that $ 2 ^{4r}\ell (K) \le  \ell (Q)$, it follows that 
we must have $ \textup{dist}(K, \partial Q) \le 20 \cdot 2 ^{r} \ell (K)$.  
By way of  contradiction, suppose that 
\begin{equation} \label{e:2L}
K \subsetneq L_1 \subsetneq L_2 \subsetneq L_3   \subsetneq Q, 
\qquad  L_1 , L_2, L_3  \in \mathcal L.  
\end{equation}
Suppose that  $ L _s  \in \widetilde {\mathcal P}_1$, for either $  s =1, 2 $. 
Since all cubes are from a grid with scales separated by $ r$, we then have $ L_s \Subset_{4r} Q$, 
and then   goodness of the parent of $ L_s$ implies 
\begin{equation*}
\lvert  L_s\rvert ^{\epsilon } \lvert  Q\rvert ^{1-\epsilon } 
\le 
\textup{dist}(L_s,\partial Q) 
\le
\textup{dist}(K, \partial Q )\le 20 \cdot 2 ^{r} \lvert  L_s\rvert
\end{equation*}
which is a contradiction. 

Thus, we must have $ L_s \not\in\widetilde {\mathcal P}_1$, 
for $ 1\leq s \leq 2$. 
Let $ Q_s \in \widetilde {\mathcal P}_1$ be such that $ L_s \in \mathcal K _{Q_s}$.  
If $ Q_1 \subsetneq Q$, then $ Q_1 \Subset_{4r} Q$, and  we  again see a con tradition to $ K\in \mathcal K_Q$, namely 
\begin{equation*}
\lvert  Q_1\rvert ^{\epsilon } \lvert  Q\rvert ^{1- \epsilon }
\le 
\textup{dist}(Q_1, \partial Q) 
\le 
\textup{dist}(K, \partial Q) 
\le 
10 \cdot 2 ^{r+1 } \lvert  Q_1\rvert .  
\end{equation*}

Assume that  $ Q_1\not\Subset_{4r} Q$.  Equality $ Q_1 = Q_2$ cannot hold, since $ K\subsetneq L_1$.   We must have 
$
Q \Subset_{4r}  Q_1 \Subset_{4r} Q_2
$, 
and that means  
\begin{equation*}
\lvert  Q\rvert ^{\epsilon } \lvert  Q _{2}\rvert^{1- \epsilon } 
\le 
\textup{dist}(Q, \partial Q _{2}) 
\le 
\textup{dist}(L_ {r+1}, \partial Q _{2})  
\le 
10\, 2 ^{r} \, \ell (Q) .  
\end{equation*}
And this final contradiction proves that \eqref{e:2L} cannot hold, and this prove the proposition.  
\end{proof}

\section{Proof of Lemma~\ref{l:2above}} 
The proof of Lemma~\ref{l:2above} depends upon a well-known case analysis, however the very weak form of the 
$ A_2$ condition\footnote{Our $ A_2$ condition is weak, but also in the case of $ d>1$, the necessary 
$ A_2$ condition comes with a power decay of $ 2d$, which is too weak to directly apply in some of the cases below.} 
complicates our analysis, and indeed, every case below requires  a new analysis.

The expression to be dominated is a sum over pairs of cubes $  2 ^{-4r} \ell (Q) \leq \ell (P) \leq 2 ^{4r} \ell (Q)$. 
Throughout,  we assume that $ \ell (Q)\le \ell (P)$, with the other case being handled by duality.  The cases are 
\begin{description}
\item[Far Apart]  $   \ell (P) \leq 2 ^{4r} \ell (Q)$ and the cubes $ 3P$ and $3Q $ do not intersect.  
\item[Surgery]   $ \ell (P) \le 2 ^{4r} \ell (Q)$ and $ 3P \cap 3Q \neq \emptyset $.  This is a delicate surgery argument, one in which we use the 
\emph{a priori} bound, and take advantage of the presence of random grids. 
\item[Nearby]  $    \ell (P) \leq 2 ^{4r} \ell (Q)$ and $ Q\subset 3P\setminus P$.   
\item[Inside]  $   \ell (P) \leq 2 ^{4r} \ell (Q)$ and $ Q\subset P$.  In this case we bound 
the sum of terms 
\begin{equation*}
 \langle  \mathsf R  (\Delta ^ \sigma  _P f \cdot (P\setminus P_Q)) , \Delta ^{w} _{Q} g  \rangle_w . 
\end{equation*}
\end{description}

\begin{lemma}[Far Apart] 
\label{l:far}
The following estimate is true:
\begin{equation}\label{e:far}
\Biggl\lvert 
\sum_{P} \sum_{\substack{Q \::\:  \ell (Q) \le \ell (P) \\ 3Q\cap 3P = \emptyset }}
\langle  \mathsf R _{\sigma } \Delta ^{\sigma } _{P} f,  \Delta ^{w} _{Q} g \rangle _{w}
\Biggr\rvert
\lesssim \mathscr R \lVert f\rVert_{\sigma } \lVert g\rVert_{w}. 
\end{equation}
The dual estimate also holds.  
\end{lemma}

\begin{proof}
Hold the relative side lengths of $ P$ and $ Q$ fixed, thus for an integer $ s\ge 0$, 
$ 2 ^{s} \ell (Q) = \ell (P)$.  For an integer $ t\ge 0$, and dyadic cube $ R$, consider the two projections 
\begin{align*}
\Pi ^{\sigma }  _{R, s, t} f&\equiv   
\sum_{\substack{P \::\: P \subset 3 ^{t+2} R \setminus  3 ^{t+1} R\\  \ell (R) = \ell (P) }}
\Delta ^{\sigma } _{P} f , 
\\
\Pi ^{w }  _{R, s, t} g&\equiv   
\sum_{\substack{Q \::\: Q\subset 3 ^{t} R\\ 2 ^{s}\ell (Q) = \ell (R) }} \Delta ^{w} _{Q} g . 
\end{align*}
Observe that it suffices to bound 
\begin{equation} \label{e:pi-suff}
\sum_{s,t=0} ^{\infty } \sum_{R\in \mathcal D} 
\bigl\lvert  \langle  
\mathsf R _{\sigma } \Pi ^{\sigma }  _{R, s, t} f , \Pi ^{w }  _{R, s, t} g
\rangle_w\bigr\rvert. 
\end{equation}
Moreover, we have 
\begin{equation} \label{e:Pi}
\sum_{R\in \mathcal D} \lVert \Pi ^{\sigma }  _{R, s, t} f\rVert_{\sigma } ^2 \lesssim \lVert f\rVert_{\sigma } ^2 
\end{equation}
and likewise for the projections $ \Pi ^{w }  _{R, s, t} f$.

Then, estimate as follows.  
\begin{align}
\sum_{\substack{Q \::\:  Q\subset 3 ^{t} R\\ \ell (Q) = \ell (R) }}  \Bigl\lvert 
\Bigl\langle  
\mathsf R _{\sigma } \Pi ^{\sigma }  _{R, s, t} f , \Delta ^{w} _{Q} g
\Bigr\rangle_{w} \Bigr\rvert 
& \lesssim \frac { 2 ^{-s}\ell (R)} {  [3 ^{t} \ell (R)] ^{d+1} } 
\sum_{\substack{Q \::\: Q\subset Q ^{(s)}\\ \ell (Q) = \ell (R) }}   
\lVert \Pi ^{\sigma }  _{R, s, t} f 
\rVert_{L ^{1} (\mathbb{R}^n;\sigma )}  \lVert\Delta ^{w} _{Q} g \rVert_{L ^{1} (\mathbb{R}^n;w)}
\\
& \lesssim 2 ^{-s} 3 ^{-t} 
\frac {  \sqrt { \sigma ( 3^{t+2} R \setminus 3 ^{t+1}R)  w (3 ^{t} R)}} {   (3 ^{t}\ell (R)) ^{d} } 
\lVert \Pi ^{\sigma }  _{R, s, t} f\rVert_{\sigma } \lVert \Pi ^{w }  _{R, s, t} g\rVert_{w} 
\\   \label{e:piPi}
& \lesssim 2 ^{-s} 3 ^{-t}  \mathscr A_2 ^{1/2} 
\lVert \Pi ^{\sigma }  _{R, s, t} f\rVert_{\sigma } \lVert \Pi ^{w }  _{R, s, t} g\rVert_{w} .  
\end{align}
The first inequality is just the standard off-diagonal kernel bound, which can be used since $ \Delta ^{w } _{Q} g$ has $ w$ integral zero, 
giving us the constant in front, which is the side length of $ Q$ times the $ L ^{\infty } $ norm of the gradient of the Riesz transform on the 
complement of $ 3 ^{t+1} R$.   
After that, apply Cauchy--Schwarz, and the $ A_2$ condition.  

By \eqref{e:Pi}, the sum over $R$ of the terms in \eqref{e:piPi} is bounded.  We have gained a geometric factor in $ s$ and $ t$, so we have the 
required bound for \eqref{e:pi-suff}. 
\end{proof}

The delicate surgery estimate is contained in the following lemma. 
It is phrased differently as the \emph{a priori} norm estimate is needed 
to complete the proof. 

\begin{lemma}[Surgery]
\label{l:surgery}  We assume that the norm  $ \mathscr N$ of $ R _\sigma $ as defined in \eqref{e:N} is finite. 
For all $ 0 < \vartheta < 1$, and choices of $ 0< \epsilon < (4d+4) ^{-1} $, 
 there is a choice of $ r$ sufficiently large so that,  uniformly in $ 0 < a_0 < b_0 < \infty $,  
\begin{equation*} 
\mathbb E \,\Biggl\lvert 
\sum_{P} \sum_{\substack{Q \::\:  2 ^{-4r} \ell (Q)\le \ell (P) \le 2 ^{4r} \ell (Q) \\ 
3P \cap 3Q \neq \emptyset  
}} \langle  \mathsf R (\sigma  \Delta ^{\sigma } _{P} f ), \Delta ^{w} _{Q} g \rangle _{w} 
\Biggr\rvert   \lesssim   \{C _{\epsilon , r, \vartheta } \mathscr R + \vartheta \tilde{\mathscr N}_0\}  \lVert f\rVert_{\sigma } \lVert g\rVert_{w} .  
\end{equation*}
\end{lemma}
\begin{proof}

It is important that the expectation over the random choice of grids appears above.  
We fix the relative lengths of $ P$ and $ Q$, 
setting $ 2 ^{s} \ell (Q) = \ell (P)$, where $ 0\le s \le r$ is fixed.  The case of $ -r\le s < 0$ is handled by duality. 
There are only a bounded number of cubes $ Q$ with length as above, such that $ 3P \cap 3Q\neq \emptyset $, and so we can assume that 
$ Q$ is a function of $ P$, but this is suppressed in the notation.   
Further, enumerating the children $ P _{i}, Q_j$, $ 1\le i,j \le n$ of $ P$, and $ Q $ respectively, we fix $ i,j$, and only consider 
\begin{equation*}
 \mathbb E  \lvert  \Delta ^{\sigma } _{P} f  \cdot \langle  \mathsf R _{\sigma }  P_i , Q_j \rangle _{w} \cdot
\Delta ^{w} _{Q} g \rvert . 
\end{equation*}
Here and below, we are suppressing the truncation parameters.

Now, we have by the testing hypothesis, uniformly over the probability space 
\begin{equation*}
\lvert  \langle  \mathsf R _{\sigma }  P_i , Q_j\cap P_i \rangle _{w} \rvert 
\lesssim \mathscr R   \sqrt{\sigma (P_i) w (Q_j) } .
\end{equation*}
In this case, Cauchy--Schwarz completes the proof, so we need only consider 
\begin{equation*}
\mathbb E \lvert   \Delta ^{\sigma } _{P} f  \cdot \langle  \mathsf R _{\sigma }  P_i , (Q_j \setminus P_i)  \rangle _{w}  \Delta ^{w} _{Q} g \rvert . 
\end{equation*}
The set $ Q_j \setminus P_i$ is decomposed into the sets $ Q _{\partial} \cup Q _{\textup{sep}}$, where 
\begin{align*}
Q _{\partial} \equiv  \{ x \in Q_j \setminus P_i \::\:  \textup{dist} (x, P_i) < \vartheta \ell  (Q)\}, \qquad 0< \vartheta <1.  
\end{align*}
This latter set depends upon $ P_i$, which is a function of the dyadic grid $ \mathcal D _{\sigma }$, holding  $\mathcal D _{w}$ fixed. 
Observe that we can estimate, using the \emph{a priori} norm inequality  \eqref{e:N}, 
\begin{align*}
\mathbb E _{\mathcal D _{\sigma }} 
\sum_{P}  
\left\lvert 
\Delta ^{\sigma } _{P} f  \cdot \langle  \mathsf R _{\sigma }  P_i , Q_ {\partial} \rangle _{w}  \Delta ^{w} _{Q} g 
\right\rvert 
& \lesssim \mathscr N \mathbb E _{\mathcal D_\sigma} 
\sum_{P}  
\left\lvert 
\Delta ^{\sigma } _{P} f  \sqrt{ \sigma  (P_i ) w( Q_ {\partial})}   \Delta ^{w} _{Q} g 
\right\rvert  
\\
& \lesssim \mathscr N\lVert f\rVert_{\sigma } 
\Biggl[  \mathbb E _{\mathcal D _{\sigma }} 
\sum_{P} 
\lvert  \Delta ^{w} _{Q} g \rvert ^2  w( Q_ {\partial})
\Biggr] ^{1/2} 
\\
& \lesssim  {\mathscr N}  \vartheta ^{1/2}  \lVert f\rVert_{\sigma } 
\lVert g\rVert_{w}. 
\end{align*}
We gain the factor of $ \vartheta ^{1/2}  $  since $ \mathbb E _{\mathcal D_\sigma} w( Q_ {\partial}) \lesssim \vartheta w (Q_j)$.  

It therefore remains to bound the term below, with cube $ P_i$ and set $ Q_ {\textup{sep}} $. 
But these sets are separated by distance $ \vartheta \ell (P) $, so that using just the kernel bound, we have 
\begin{align*}
\lvert  \langle  \mathsf R _{\sigma }  P_i , Q_ {\textup{sep}} \rangle _{w} \rvert 
\leq  C _{\vartheta, r } \frac {\sigma (P_i) w (Q_j)} { \ell (P) ^{d}} \lesssim  C _{\vartheta } \mathscr A_2 ^{1/2} \sqrt {\sigma (P_i) w (Q_j)}, 
\end{align*}
and this is clearly enough to complete the proof. 

\end{proof}

\begin{lemma}[Inside Term]
\label{l:inside}  There holds 
\begin{equation*}
\sum_{P} \sum_{\substack{Q \::\:  Q\Subset _{4r} P_Q}} 
\lvert  \langle  \mathsf R _{\sigma } (\Delta ^{\sigma } _{P} f \cdot (P \setminus P_Q)), \Delta ^{w} _{Q} g \rangle _{w}  \rvert
   \lesssim \mathscr R  \lVert f\rVert_{\sigma } \lVert g\rVert_{w} .  
\end{equation*}
\end{lemma}

\begin{lemma}[Nearby Term]
\label{l:nearby}  There holds 
\begin{equation*}
\sum_{P}  
\sum_{\substack{Q \::\:  2 ^{4r} \ell (Q)\le \ell (P)  \\ 
Q\subset 3P\setminus P   }}
\lvert \langle  \mathsf R _{\sigma } \Delta ^{\sigma } _{P} f , \Delta ^{w} _{Q} g \rangle _{w}  \rvert 
   \lesssim \mathscr R  \lVert f\rVert_{\sigma } \lVert g\rVert_{w} .  
\end{equation*}
\end{lemma}

Both are argued by a similar method, and we present the proof of the nearby term. 
The key points are the goodness of cubes and the $ A_2$ condition. 
These two proofs are the only place in which the full strength of the $ A_2$ condition 
is used. Moreover, the $ A_2$ bound has  to be arranged correctly to overcome certain dimensional 
obstructions not present in the Hilbert transform case.  

\begin{proof}
We hold the relative lengths of $ Q$ and $ P$ to be  fixed by an integer $ s\geq 4r$.
By a crude application of the monotonicity principle, there holds for each child $ P'$ of $ P$, 
\begin{equation*}
\lvert \langle  \mathsf R _{\sigma } (\Delta ^{\sigma } _{P}  f \cdot P') , \Delta ^{w} _{Q} g \rangle _{w}  \rvert  
\lesssim 
[\, \lvert \Delta ^{\sigma } _{P}  f \rvert\,] _{P'} ^{\sigma } 
\mathsf P ^{g} _{\sigma } (P',Q) w (Q) ^{1/2} \lVert \Delta ^{w} _{Q} g \rVert_w  .
\end{equation*}
Clearly, we will use Cauchy-Schwarz in $ Q$.  
To organize the sum of the Poisson related terms, we first estimate 
\begin{equation*}
\mathsf P ^{g} _{\sigma } (P',Q) ^2  w (Q)  
\lesssim \sigma (P') \int _{P'} \frac { \ell (Q) ^2 } { \ell (Q) ^{2d+2} + \lvert  x-x_Q\rvert ^{2d+2} } \;  \sigma (dx)  \cdot 
w (Q). 
\end{equation*}

Now, $ Q$ is good, and $ 2 ^{s} \ell (Q)=\ell (P)$, 
hence for each $ x\in P'$, $ \lvert  x-x_Q\rvert \gtrsim 2 ^{- \epsilon s} \ell (P) $.  Let $ \mathcal R_P^s$ be a collection of cubes $ R \subset 3P \setminus P$ such that 
 (1) each has side length 
$ \ell (R) \simeq  2 ^{- \epsilon s}\ell (P)  $,  (2) 
$ \textup{dist}(R, P) \gtrsim 2 ^{- \epsilon s}\ell (P)  $,  (3) 
each good $Q$ with $  2 ^{s}\ell (Q)$ is contained in some $ R\in \mathcal R _{P} ^{s}$, and (4) 
the cardinality of $ \mathcal R_P ^{s}$ is at most $ C_n 2 ^{n \epsilon s}$. 

Then, for each $ R\in \mathcal R_P ^{s}$, we can use the $ A_2$ bound. 
\begin{align*}
\sum_{\substack{Q \::\:  2 ^{s} \ell (Q)= \ell (P)  \\ 
Q\subset R  }} &
\int _{P'} \frac { \ell (Q) ^2 } { \ell (Q) ^{2d+2} + \lvert  x-x_Q\rvert ^{2d+2} } \;  \sigma (dx) \cdot 
w (Q) 
\\& \lesssim 
2 ^{- 2(1-\epsilon)s} 
\int _{P'} \frac { 1 } { \ell (R) ^{2d} + \lvert  x-x_R\rvert ^{2d} } \; \sigma (dx)  \cdot 
w (R) 
\lesssim 2 ^{-2 (1- \epsilon )s} \mathscr A_2 . 
\end{align*}

Combining these observations, we have 
\begin{align*}
\sum_{R\in \mathcal R_P ^{s}} 
\sum_{\substack{Q \::\:  2 ^{s} \ell (Q)= \ell (P)  \\ 
Q\subset R  }}  &
\lvert \langle  \mathsf R _{\sigma } (\Delta ^{\sigma } _{P}  f \cdot P') , \Delta ^{w} _{Q} g \rangle _{w}  \rvert   
\\& \lesssim   2 ^{-2 (1- \epsilon )s} \mathscr A_2  ^{1/2} 
[\, \lvert \Delta ^{\sigma } _{P}  f \rvert\,] _{P'} ^{\sigma }  \sigma (P') ^{1/2} 
\sum_{R\in \mathcal R_P ^{s}}  
\Biggl[
\sum_{\substack{Q \::\:  2 ^{s} \ell (Q)= \ell (P)  \\ 
Q\subset R  }} \lVert \Delta ^{w} _{Q} g \rVert _{w}^2 
\Biggr] ^{1/2} 
\\
& \lesssim 
 2 ^{-2 (1- \epsilon -n \epsilon /2 )s} \mathscr A_2  ^{1/2}  
 [\, \lvert \Delta ^{\sigma } _{P}  f \rvert\,] _{P'} ^{\sigma }  \sigma (P') ^{1/2}  
 \Biggl[
\sum_{\substack{Q \::\:  2 ^{s} \ell (Q)= \ell (P)  \\ 
Q\subset 3P \setminus P   }} \lVert \Delta ^{w} _{Q} g \rVert _{w}^2 
\Biggr] ^{1/2} .  
\end{align*}
Notice that we have gained a factor comparable to the square root of the cardinality of $ \mathcal R ^{s}_P$. 
We take $ 0< \epsilon < 1$ so small that the exponent on $ s$ above is positive. 
A further sum over $ P$, children $ P'$,  and $ s\ge 4r$ are very easy to complete. This finishes the proof. 

\end{proof}

\end{document}